\documentclass[a4paper]{amsart}
  \usepackage{amsthm}
  \usepackage{amssymb}
  \usepackage{amsmath}
  \usepackage{eucal}
  \usepackage{mathrsfs}
  \usepackage[all]{xy}
  \usepackage{graphicx}
  \usepackage{bbm}
\usepackage[dvipsnames,usenames]{color}
\newcommand{\R}{\mathbb{R}} 
\newcommand{\Z}{\mathbb{Z}} 
\newcommand{\N}{\mathbb{N}} 
\newcommand{\Bcal}{\mathcal{B}}

\newcommand{\Ccal}{\mathcal{C}}  
\newcommand{\Dcal}{\mathcal{D}}

\newcommand{\Tcal}{\mathcal{T}}

\newcommand{\Zrm}{\mathrm{Z}}  
\newcommand{\f}{\mathrm{f}}

\newcommand{\urm}{\mathrm{u}}

\newcommand{\zrm}{\mathrm{z}}

\newcommand{\ub}{{\bf\urm}} 
\newcommand{\fb}{{\bf\f}}
\newcommand{\zb}{{\bf\zrm}} 
\newcommand{\Zb}{{\bf\Zrm}} 

\newcommand{\Frac}{\displaystyle\frac}
\newcommand{\Int}{\displaystyle\int}

\newcommand{\Sum}{\displaystyle\sum}

\newcommand{\bbeett}{\partial_x^\beta}
\newcommand{\bbtt}{\beta}
\newcommand{\Dt}{\Delta t}
\newcommand{\Dx}{\Delta x}
  
  \newtheorem{df}{Definition}[section]
  \newtheorem{teo}[df]{Theorem}
  \newtheorem{prop}[df]{Proposition}

  \theoremstyle{definition}
  \newtheorem{remark}[df]{Remark}

\theoremstyle{definition}
  
 \theoremstyle{definition}
  \newtheorem{example}[df]{Example}
\begin{document}

\title[\sc TAHO Schemes for Dissipative  BGK Hyperbolic Systems]{Time Asymptotic High Order Schemes for Dissipative BGK Hyperbolic Systems}
%
\author[\sc D. Aregba-Driollet, M. Briani, and R. Natalini]{Denise
  Aregba-Driollet$^{1}$, Maya Briani$^{2}$,  and Roberto
  Natalini$^{2}$}\thanks{\,$^{1}$IMB, UMR CNRS 5251,  Universit\'{e}
  de Bordeaux.\\
$^{2}$ Istituto per le Applicazioni del Calcolo ``Mauro Picone'', Consiglio Nazionale delle Ricerche.  
}



\date{}


\begin{abstract}
{We introduce a new class of finite differences schemes to approximate one dimensional dissipative semilinear hyperbolic systems with a BGK structure.
Using precise analytical time-decay estimates of the local truncation error, it is possible to design schemes, based on the standard upwind approximation, which are   increasingly accurate for large times when approximating small perturbations of constant asymptotic states. Numerical tests show  their better performances with respect to those of other schemes. 
}
\end{abstract}

\maketitle

\section{Introduction}
Consider a  BGK  system in one space dimension, for the unknowns $f^i\in\R^k$, $k\geq 1$ and $i=1,...,m$:
\begin{equation}\label{sysdiag}
\left\{\begin{array}{l}
\partial_t f^i+\lambda_i \partial_x f^i=M_i(u)-f^i, \\
\\
\hbox{ where } u:=\sum_{i=1}^m  f^i.
\end{array}\right.
\end{equation}
Here $x\in\R$ and $t>0$, the $\lambda_i$, for $i=1,...,m$, are   given distinct real values, and the functions $M_i=M_i(u)\in\R^k$ are smooth functions of $u$ such that: 
\[\sum_{i=1}^m M_i(u)=u.\]
Following \cite{BHN}, we rewrite system \eqref{sysdiag} in its {\it conservative-dissipative} form. 
This  means that we assume that there exists an invertible matrix
\begin{equation}\label{matrixD}
D=\left(\begin{array}{cc} D_{11}& D_{12} \\ D_{21}& D_{22}\end{array}\right),
\end{equation}
such that, setting  $m_{1}=k$, $m_2=k(m-1)$, the new unknown
\begin{equation}\label{changeVarGen}
 Z = D f =(u, \tilde Z)^T\in\R^{m_1}\times\R^{m_2},
\end{equation}
 solves the system
\begin{equation}\label{syscd1}
\left\{
\begin{array}{l}
\partial_t u + A_{11}\partial_x u + A_{12}\partial_x \tilde Z=0, \\
\\
\partial_t \tilde Z + A_{21}\partial_x u + A_{22}\partial_x \tilde Z 
= \tilde Q(u) - \tilde Z,
\end{array}\right.
\end{equation}
where $A$ is symmetric and $\tilde Q(u)$ is quadratic in $u$, i.e.:
$\tilde Q(0) =0$ and $ \tilde Q'(0)=0$. Observe that, after the
transformation, the source term is zero in the first component and the
second one is the sum of a quadratic term and of the dissipative term $-\tilde Z$.

It is proved in \cite{HaNa03} and \cite{BHN} that, under some additional conditions, usually induced by suitable entropy functions,  and for initial data which are small perturbations of constant equilibrium states and smooth in some suitable norms, the corresponding smooth solutions exist globally and their  $L^\infty$-norm decay,  for large times, as 
\begin{equation}\label{decaybase}
  u = O( t^{-1/2}),\quad \tilde Z=O( t^{-1}),
\end{equation}
and similar estimates are available for their space and time derivatives. Notice that the improved estimate for the unknown $\tilde Z$ can only be obtained in these new coordinates and does not hold for other combinations of the unknowns. 

The aim of this paper it to take advantage by these time decay estimates to build up more accurate numerical schemes. To be more specific, we  show in the following that for standard numerical schemes, for instance upwind schemes with the source term approximated pointwise by the standard Euler scheme, the local truncation error for the conservative-dissipative unknowns $(u,\tilde Z)$ has the following decay as $t\rightarrow +\infty$, for a fixed CFL ratio:
\[\label{trunc}\begin{array}{cc}
\Tcal_u(x,t)=   O(\Dx\ t^{-3/2}) , & 
\Tcal_{\tilde Z}(x,t)= O(\Dx\  t^{-3/2}).
\end{array}\]
It can be seen numerically that the corresponding absolute errors, for a fixed space step,  decays as 
\[ e_u(t)= O( t^{-1/2}), \ e_{\tilde Z}(t)= O( t^{-1}),\]
which implies, taking into account \eqref{decaybase},  that the relative error is essentially constant in time.

Here, our main goal is to improve the decay rate of the truncation error to achieve an effective decay in time of the relative error, both in $u$ and $\tilde Z$. Before presenting our strategy and our main results, let us review some different attempts to design effective numerical approximations for hyperbolic  equations with a source term.
Let us mention some families of schemes, sometimes overlapping: 
Well Balanced  \cite{GrLe96, Go00, Go02, GJ02, GT03, Ji01, bou}, Runge-Kutta IMEX  \cite{PR03}, upwinding source \cite{Roe, BV94, BOP07,DMR}, and asymptotic preserving \cite{Ji99, Ji10}. The main idea in all these schemes is to use some knowledge of the actual time behavior of the solutions to improve their numerical approximation, at least in some specific regimes. 

In particular, in \cite{DMR}, the linear version of the present
problem was considered and therefore, to approximate the solutions
around non constant asymptotic states, some  schemes were proposed,
which had the property to become higher order (in space) for large
times, thanks to the careful consideration of the analytical decay
rates of the solutions. A different attempt was given by the well
balanced schemes, see for instance \cite{GrLe96, Go00, bou}, namely
schemes which are exact when computed on stationary solutions of the
problem, even if up to now, the time decay rate of the unsteady
solutions has  not yet been explicitely considered. However, the Asymptotic Preserving properties of some Well Balanced schemes, can yield nice results for large times, as in \cite{GT03}, see Section \ref{Sec5} below.

In the present case, a striking difference with these previous works
lies in the fact that  the asymptotic equilibrium states are constant and therefore all the consistent schemes are {\em exact} on them. So, the goal of our work is a bit different. In this paper, we  design schemes which are able to improve their performance for large times, when the initial data are small perturbations of a given constant equilibrium state. To obtain these results, we use the estimates in \cite{BHN} to perform a detailed analysis of the behavior of the truncation error for a general class of schemes, which  generalize and improve those introduced in \cite{DMR}. Thanks to this analysis, we are able construct schemes such that the truncation order behaves as
\[\label{trunc2}\begin{array}{cc}
\Tcal_\ub(x,t)=   O(\Dx\ t^{-2}), & 
\Tcal_{\tilde Z}(x,t))= O(\Dx\ t^{-2}), 
\end{array}\]
for a fixed CFL ratio and such that their asymptotic numerical error, observed in the practical tests, improves of $t^{-1/2}$ on the previous schemes. 

The plan of the paper is the following. In Section \ref{Sec2}, we
introduce our analytical framework.  The main schemes are derived   in
Section \ref{Sec3}, where we show how to improve the time decay of
their local truncation error. Sections \ref{sec_monot2x2} and
\ref{3x3case} are devoted to the monotonicity conditions for the new
scheme in the $2\times 2$ and  $3\times 3$ cases respectively. Then we
present some remarks in the linear $2\times 2$ case, to allow a direct  comparison with other schemes. Section \ref{Sec5} presents some numerical tests which show the nice behavior of our new schemes both in the linear and the nonlinear cases.

\section{The analytical framework}\label{Sec2}

Let us observe that when transforming system \eqref{sysdiag} in system \eqref{syscd1}, we can always assume that the block $D_{11}$ and $D_{12}$ have the special form 
$$D_{11}=I_{k}, \quad D_{12}=(I_k I_k \cdots I_k)\in \R^{k\times m_2},$$
and setting $\Lambda=diag(\lambda_1 I_k ,...,\lambda_m I_k)$, we have that
$$A=\left(\begin{array}{cc} A_{11}& A_{12} \\ A_{21}& A_{22}\end{array}\right)=D\Lambda D^{-1}.$$
Therefore, we can rewrite our system in a more compact  form:
\begin{equation}\label{syscd}
\partial_t Z + A\partial_x Z = -Z + DM(u).
\end{equation}

To guarantee the existence of the matrix $D$ in \eqref{matrixD}, we can assume that our system is strictly entropy dissipative in the sense of \cite{HaNa03} and verifies the Shizuta-Kawashima condition \cite{ShKa84,HaNa03,BHN}. For instance, following Bouchut \cite{Bou99},  we may assume the following sufficient condition.

{\bf Condition ED (Entropy Dissipation condition)} {\it There exists an open set $\Omega \subseteq \R^k$ and a strictly convex function $\eta=\eta(u): \Omega\to \R$, such that the matrix ${M'_i}(u)^T\, \eta''(u)$ is symmetric and strictly positive defined for all $u\in \Omega$ and $i=1,\dots,m$. }\\

Under this condition, and using the results in \cite{Bou99, HaNa03, BHN}, it is possible to prove the existence of a matrix $D$ in \eqref{matrixD}, with all the properties mentioned above. The details of the derivation of the actual coefficients of this matrix, which in general is not unique, but depends on the entropy dissipative function, can be found in the general case in \cite{BHN}, and are not relevant for the following discussion, even if of course they are  relevant  to derive the final schemes. However, in the examples presented below, the matrix $D$ is always explicitly given.

\subsection{Algebraic conditions}

 Let us state now some general properties for the matrix $D$ that we shall use in the following, and which are easily derived under {\bf condition  ED}.

First of all, since $A=D\Lambda D^{-1}$ in \eqref{syscd} is symmetric, we have that the matrix  $H:=D^TD$ and the diagonal matrix $\Lambda$ commute, i.e.
\begin{equation}\label{commute}
 D^TD\Lambda = \Lambda D^TD,
\end{equation}
So, for the generic $ij$ block element of the matrix $H$, we have that
$ (\lambda_i-\lambda_j)H_{ij}=0$, hence
\[ H_{ij}=0 \mbox{ for } i\ne j \mbox{ once } \lambda_i\ne \lambda_j.\]
Since $\Lambda$ is a block diagonal matrix with $m$ distinct eigenvalues with multiplicity egual to $k$, we get that $H$ is a block diagonal matrix of the following form
\begin{equation}\label{matrixH}
 H=D^TD= diag(h_1,...,h_m),\ h_i\in\R^{k\times k}, \ i=1,..,m.
\end{equation}
Moreover, the  matrix $H=D^T D$ is symmetric and invertible and  
we can find the following expression for $D^{-1}$, 
\begin{equation}\label{Dinv}
 D^{-1} = H^{-1} D^T = \left(\begin{array}{cc} H_1^{-1} & H_1^{-1}D_{21}^T \\ H_2^{-1}D_{12}^T & H_2^{-1}D_{22}^T\end{array}\right),
\end{equation}
where $H_1=h_1$ and $H_2=diag(h_2,..,h_m)$. Finally, using  \eqref{matrixH}, the following relations hold:
\begin{equation}\label{Drelaz}
	D_{12} = - D_{21}^T D_{22}, \quad D_{21}^TD_{21}=H_1-I_k, \quad D_{12}^TD_{12}+D_{22}^TD_{22} = H_2.
\end{equation}
There is another important relation involving the matrix $H$, which will be of use in the following. Let $A_{11}$ be the top-left block of the symmetric matrix $A$ in system \eqref{syscd}. Therefore: $A_{11}=\Lambda_1 H_1^{-1}+D_{12}\Lambda_2H_2^{-1}D_{12}^T$, i.e.
\begin{equation}\label{S}
A_{11}=\lambda_1 h_1^{-1}+\ \Big(I_k I_k \cdots I_k\Big)\  diag(\lambda_2 h_2^{-1},...,\lambda_m h_m^{-1})\ \left(\begin{array}{c} I_k \\ I_k \\ \vdots\\ I_k\end{array}\right)=\Sum_{i=1}^{m} \lambda_i h_i^{-1}.
\end{equation}

From now on, we shall consider each matrix $N\in\R^{km \times km}$ to
be decomposed into blocks as follows:
\begin{equation}\label{blockDim}
N=\left(\begin{array}{cc} N_{11}& N_{12} \\ N_{21}& N_{22}\end{array}\right), 
\end{equation}
with $N_{11}\in\R^{k\times k},\ N_{12}\in\R^{k\times m_2},\ N_{21}\in\R^{m_2 \times k}, N_{22}\in\R^{m_2\times m_2}$.
For the particular case of a block diagonal matrix
$N=diag(n_1,...,n_m)$, $n_i\in\R^{k\times k}$ for $i=1,...,m$, we
shall use the following notation
\begin{equation}\label{blockDiagDim}
N=diag(N_1,N_2), \quad N_1=n_1 \mbox{ and } N_2=diag(n_2,..,n_m).
\end{equation}
Moreover, for a generic vector $V\in\R^{km}$ we shall write
\begin{equation}\label{vectDecDim}
 V = (V_1 ,\tilde V ), \hbox{ with} \quad V_1\in\R^{k}, \quad \tilde V\in\R^{m_2}.
\end{equation}
\begin{example}\label{example2x2}

Consider the special case of the $2\times 2$  hyperbolic Jin-Xin relaxation system \cite{JX95, Na96}:
\begin{equation}\label{basic}
\left\{\begin{array}{l}
\partial_t u+ \partial_x v=0, \\
\\
\partial_t v+ \lambda^2 \partial_x u=F(u)-v,
\end{array}\right.
\end{equation}
for $\lambda>0$, where the unknowns $u$ and $v$ are scalar and the function $F=F(u)$ is smooth, with $F(0)=0$. 
This case is obtained from \eqref{sysdiag} for $k=1$, $m=2$, and  $\lambda_2=-\lambda_1=\lambda$, by setting 
\[ u=f^1+f^2, \ v=\lambda(f^2-f^1), \ F(u)=\lambda(M_2-M_1).\]

Recall that for  bounded initial data, and for $\lambda>\tilde M$, where $\tilde M$ is a positive constant which depends on $F$ and on the initial data, there exists a global bounded solution to  the Cauchy problem \eqref{basic}, see \cite{Na96}. Under the weaker condition
\begin{equation}\label{sub}
\lambda>|F'(0)|, 
\end{equation}
the problem is  dissipative in the sense of \cite{HaNa03} and the Shizuta-Kawashima condition is verified, at least for small values of $u$, and so, at least for smooth and small initial data, there exists again a global smooth solution to the same problem.  To obtain the time decay rates of these solutions,  we need to rewrite the problem in the {\em conservative--dissipative} coordinates. Assume \eqref{sub} and set $a=F'(0)$  so that the quantity $\mu=(\lambda^2-a^2)^{-1/2}$ is  real and positive. Setting $\tilde Z=\mu (v-a u)$, the new unknown $(u, \tilde Z)$ solves the problem in form \eqref{syscd1}:
\begin{equation}\label{cv_sys}
\left\{ \begin{array}{l}
\partial_t u+\partial_x (au+\frac{1}{\mu}\tilde Z)=0,\\
\\
\partial_t \tilde Z+\partial_x (\frac{u}{\mu}-a\tilde Z)=\mu(F(u)-au)-\tilde Z.
\end{array}\right.
\end{equation}
In this case the matrix $D$ is given by 
\[ D =  \left(\begin{array}{cc}
	1 & 1 \\
	-\mu a_+ & \mu a_-
	\end{array}\right),
\]
where  $a_\pm=\lambda\pm a>0$, from assumption \eqref{sub}.
\end{example}

\begin{example}\label{example3x3}
Let us now compute the conservative-dissipative form for the following $3\times 3$ BGK model, 
\begin{equation}\label{BGK3}
\left\{
\begin{array}{l}
\partial_t f_1 - \lambda \partial_x f_1= M_1(u)-f_1 , \\
\partial_t f_2=M_2(u)-f_2 , \\
\partial_t f_3 + \lambda \partial_x f_3=M_3(u)-f_3 .
\end{array}
\right.
\end{equation}
Let $F=F(u)$ be a smooth scalar function such that $F(0)=0$ and let $\gamma$ be such that 
$\gamma^\prime(u)=|F^\prime(u)|$, with $\gamma(0)=0$. We choose our three maxwellian functions as follows, for $\beta\in]0,1[$ and $\lambda>0$ 
\begin{equation}\label{m3}
\begin{array}{l}
M_1(u)=\frac{1}{2}\left( \frac{\gamma(u)-F(u)}{\lambda}+\beta u\right), \  M_3(u)=\frac{1}{2}\left( \frac{\gamma(u)+F(u)}{\lambda}+\beta u\right),\\
\\ 
M_2(u)=u-M_1(u)-M_3(u)=(1-\beta)u- \frac{\gamma(u)}{\lambda}.
\end{array}
\end{equation}
The functions $M_i$, $i=1,2,3$, are strictly increasing if for any $u$ under consideration
\begin{equation}\label{mono3}
\lambda > \frac{|F^\prime(u)|}{1-\beta},
\end{equation}
and so {\bf condition ED} is verified. Let $a=F^\prime(0)$ and $\alpha=|a|+\beta \lambda$, following the results in \cite{Bou99, HaNa03, BHN} it is possible to compute the matrix $D$ for the change of variables \eqref{changeVarGen} as
\[
D=\left(
\begin{array}{ccc}
1 & 1 & 1  \\ \\
\frac{\alpha+a}{\alpha-a}\sqrt{\frac{\lambda(\alpha-a)}{\alpha(\alpha+a)}} & 0 & -\sqrt{\frac{\lambda(\alpha-a)}{\alpha(\alpha+a)}} \\ \\
-\sqrt{\frac{\lambda-\alpha}{\alpha}} & -\sqrt{\frac{\lambda-\alpha}{\alpha}}+ \frac{\lambda}{\sqrt{\alpha(\lambda-\alpha)}} & -\sqrt{\frac{\lambda-\alpha}{\alpha}}
\end{array}
\right).
\]
\end{example}

\vskip 10pt

\subsection{Time Decay Properties }\label{decays}
Here we report some time decay results which have been mainly proved in \cite{BHN}, and which will be useful in the following. 
\begin{teo}\label{teo_decay_u}
Let $Z=(u, \tilde Z)$ be the local smooth  solution to the Cauchy
problem for system \eqref{syscd1}. Let
$E_s=\max{\{\|Z(0)\|_{L^1},\|Z(0)\|_{H^s}\}}$ a norm for the initial
data, which are taken in $H^k$ for $k$ large enough, and assume
$E_{2}$ small enough. Therefore, under the Condition ED, the solution
is global in time and the  following decay estimate holds, for all
$\beta \leq k$:
\begin{equation}\label{decay_u}
\| \bbeett  Z(t)\|_{L^\infty} \leq C \min{\{1,t^{-\frac{1}{2}-\frac{\bbtt}{2}}\}} E_{\bbtt+1},
\end{equation} 
with $C=C(E_{\bbtt+\sigma})$ for $\sigma$ large enough. For the dissipative part $\tilde Z$ we have, under the same conditions,  the more precise estimate
\begin{equation}\label{decay_z}
\| \bbeett  \tilde Z(t)\|_{L^\infty} \leq C \min{\{1,t^{-1-\frac{\bbtt}{2}}\}} E_{\bbtt+1},
\end{equation} 
for another constant  $C=C(E_{\bbtt+\sigma})$ as previously.
\end{teo}

\begin{teo}\label{decay_t}
Under the assumptions of Theorem \ref{teo_decay_u}, we have the following decay estimates for the time derivatives of $Z$:
\begin{equation}\label{decay_ut}
\| \bbeett  Z_t(t)\|_{L^\infty} \leq C \min{\{1,t^{-1-\frac{\bbtt}{2}}\}} E_{\bbtt+2};
\end{equation} 
\begin{equation}\label{decay_zt}
\| \bbeett  \tilde Z_t(t)\|_{L^\infty} \leq C \min{\{1,t^{-\frac{3}{2}-\frac{\bbtt}{2}}\}} E_{\bbtt+3};
\end{equation} 
\begin{equation}\label{decay_utt}
\| \bbeett  Z_{tt}(t)\|_{L^\infty} \leq C \min{\{1,t^{-\frac{3}{2}-\frac{\bbtt}{2}}\}} E_{\bbtt+3};
\end{equation} 
with $C=C(E_{\bbtt+\sigma})$ for $\sigma$ large enough.
\end{teo}

\begin{proof}[Proof of Theorem \eqref{decay_t}]
We have only to prove inequality \eqref{decay_utt}, since   the other ones are proved in \cite{BHN}. First we have that
\begin{equation}\label{decay_uttbis}
\partial_{tt} u =  A_{11}^2\partial_{xx} u + A_{11} A_{22} \partial_{xx} \tilde \Zb - A_{12} \partial_{tx} \tilde \Zb,
\end{equation}
and so, using inequalities  \eqref{decay_z}, \eqref{decay_ut} and \eqref{decay_zt}, we have \eqref{decay_utt} for the conservative part. Now, using the second part of system \eqref{syscd1}, we have 
\begin{equation}\label{decay_ztt}
\begin{array}{lcl}
\partial_{tt} \tilde \Zb &=& - A_{21}\partial_{tx}u-A_{22}\partial_{tx}\tilde \Zb+ \tilde Q'(u)\partial_{t}u-\partial_t \tilde\Zb,
\end{array}
\end{equation}
which yields the proof, using that the term $\tilde Q(u)$ is at least quadratic in $u$.
\end{proof}

\begin{remark}\label{remarkA0} Consider the special case of system \eqref{syscd1} when $ A_{11}=0$. We can improve the decay of the time derivative of the unknowns. Actually
\[\|  u_t(t)\|_{L^\infty} =\|  A_{12}\partial_x \tilde \Zb\|_{L^\infty} \leq C \min{\{1,t^{-\frac{3}{2}}\}} E_{ 3}.\]
Then, we  have
\begin{equation}\label{decay_utt_a0}
\|   \partial_{tt} u (t)\|_{L^\infty} =\|    A_{12} \partial_{tx} \tilde \Zb\|_{L^\infty} \leq C \min{\{1,t^{-2}\}} E_{4},
\end{equation}
\begin{equation}\label{decay_uxt_a0}
\|   \partial_{tx} u(t)\|_{L^\infty} =\|  A_{12} \partial_{xx} \tilde \Zb\|_{L^\infty} \leq C \min{\{1,t^{-2}\}} E_{4}.
\end{equation}
Finally, for the second derivative of the dissipative part, we have
\begin{equation}\label{ztta0}
\begin{array}{lcl}
\partial_{tt} \tilde \Zb &=& - A_{21}\partial_{tx}u-A_{22}\partial_{tx}\tilde \Zb+ \tilde Q'(u)\partial_{t}u-\partial_t \tilde\Zb.
\end{array}
\end{equation}
Therefore, we can apply the Duhamel's formula to obtain
$$\partial_t \tilde \Zb = e^{-t}\partial_t \tilde\Zb_0 -\Int_0^t
e^{-(t-s)}\left(A_{21}\partial_{tx}u(s)+A_{22}\partial_{tx}\tilde\Zb(s)+
  Q'(u) \partial_tu(s)\right)ds.$$
Now, since the function $Q$ is quadratic in $u$ and using the previous estimates, we find that
\[
\|   \partial_{t}\tilde \Zb \|_{L^\infty} \leq C \left(e^{-t}+\Int_0^t e^{-(t-s)} \min{\{1,s^{-2}\}} ds\right) \leq C \min{\{1,t^{-2}\}}.
\]
By plugging this last inequality in \eqref{ztta0} we obtain 
\begin{equation}\label{decay_ztt_a0}
\|  \partial_{tt} \tilde \Zb \|_{L^\infty}+\|  \partial_{t} \tilde \Zb \|_{L^\infty}\leq C \min{\{1,t^{-2}\}}.
\end{equation}
\end{remark}

\section{The Numerical Approximation}\label{Sec3}
In this section we first introduce general finite difference approximations for  
system (\ref{sysdiag}). Then, we compute the local truncation error of these schemes and we discuss 
its decays properties. The main result is given in Theorem \ref{prop_aho_gen}, where a class of TAHO schemes is fully  characterized. First, we approximate the differential part following 
the direction of the characteristic velocities, so  we
study  the methods for the system in diagonal form
(\ref{sysdiag}).

 We denote by $f=(f^1,...,f^m)$ the exact solution. Let  $\Dx$ the uniform mesh-length  and  $x_j=j\,\Dx$   the spatial grid points for all $j\in \Z$. The time levels $t_n$, with $t_0=0$, are also spaced uniformly with mesh-length $\Dt=t_{n+1}-t_n$
for $n\in\N$. We denote by $\rho$ the CFL ratio $\rho=\Dt/\Dx$, which is 
taken constant through all the paper. 

We consider the Cauchy problem for system 
(\ref{sysdiag}) possibly subjected to some stability conditions. The initial data $ f^0$ is supposed to be smooth
and  approximated by its node values.
The approximate solution $(f^1_{j,n},...,f^m_{j,n})^T$, $f^i_{j,n}\in\R^{m_1}$, $i=1,...,m$, for $j\in\Z$ and $n\in\N$, is given by
\begin{equation}\label{numdiag}
\begin{array}{l}
\Frac{f^i_{j,n+1}-f^i_{j,n}}{\Dt}+\Frac{\lambda_i}{2\Dx}\left(f^i_{j+1,n}-f^i_{j-1,n}\right)-\Frac{q_i}{2\Dx}\delta_x^2 f^i_{j,n}\\ \\\qquad=\Sum_{l=-1,0,1}\left(\Bcal^i_{l}(u_{j+l,n})-\beta^i_l f^i_{j+l,n}\right),
\\ \\
f^i_{j,0}= \fb^i_0(x_j), \quad  j\in\Z,
\end{array}
\end{equation} 
where $\delta_x^2 f_{j,n}=(f_{j+1,n}-2 f_{j,n}+f_{j+1,n})$, for all $i=1,...,m$. The artificial diffusion terms
$q_i$ are diagonal matrices in $\R^{{m_1}\times {m_1}}_+$. The source term approximation is defined, for $l=-1,0,1$, by 
the diagonal matrices $\beta^i_l\in\R^{{m_1}\times {m_1}}$ and by the vectors of functions $\Bcal^i_l(\cdot)\in\R^{m_1}$.

We assume the scheme (\ref{numdiag}) is consistent with system (\ref{sysdiag}),
i.e, for all $i=1,...,m$ 
\begin{equation}\label{consistency}
\begin{array}{l}
\beta^i_{-1}+\beta^i_{0}+\beta^i_{1}=I_{m_1}+\Dx C^i,\\
\\
\Bcal^i_{-1}(u)+\Bcal^i_0(u)+\Bcal^i_1(u) = M_i(u)+\Dx \Ccal_i(u),
\end{array}
\end{equation}
where $C^i =diag{(c^i_1,...,c^i_{m_1})}\in\R^{m_1\times m_1}$ and $\Ccal_i(u)$ are $m_1$ functions to be defined.

By applying the change of variables \eqref{changeVarGen}, the scheme applies to the system in the conservative-dissipative form \eqref{syscd1}, and it reads
\begin{equation}\label{numcd}
\begin{array}{l}
\Frac{Z_{j,n+1}- Z_{j,n}}{\Dt}+\Frac{A}{2\Dx}\left(Z_{j+1,n}-Z_{j-1,n}\right)
-\Frac{\bar Q}{2\Dx}\delta_x^2 Z_{j,n}\\
\\
=\Sum_{l=-1,0,1}\Big(\bar \Bcal_{l}(u_{j+l,n})- \bar b_l Z_{j+l,n}\Big),
\end{array}
\end{equation}
where $Q=diag(q_1,...,q_m)$, $\bar Q = D Q D^{-1}$ and for $l=-1,0,1$, 
\[
\begin{array}{l}
b_l =diag(\beta^{1,1}_l,...,\beta^{1,k}_l,...,\beta^{m,1}_l,...,\beta^{m,k}_l)\in\R^{km\times km},\quad  
\bar b_l=D b_l D^{-1}\in \R^{km\times km},\\
\\
\Bcal_l(u) = (\Bcal^{1,1}_l(u),...,\Bcal^{1,k}_l(u),...,\Bcal^{m,1}_l(u),...,\Bcal^{m,k}(u))^T\in\R^{km}, \\
\bar \Bcal_l =D\Bcal_l(u) \in \R^{km}.
\end{array}
\]
By separating the system with respect to the two variables $u$ and $\tilde Z$, we get
\begin{equation}\label{numcd_uz}
\begin{array}{l}
\Frac{u_{j,n+1}-u_{j,n}}{\Dt}+\Frac{A_{11}}{2\Dx}\left(u_{j+1,n}-u_{j-1,n}\right)+\Frac{A_{12}}{2\Dx}\left(\tilde Z_{j+1,n}-\tilde Z_{j-1,n}\right) \\
\\
-\Frac{\bar Q_{11}}{2\Dx}\delta_x^2 u_{j,n}-\Frac{\bar Q_{12}}{2\Dx}\delta_x^2 \tilde Z_{j,n}=\Sum_{l=-1,0,1}\left(\bar\Bcal^{1}_{l}(u_{j+l,n})-\bar b^{11}_l u_{j+l,n}-\bar b^{12}_l \tilde Z_{j+l,n}\right), \\
\\
\Frac{\tilde Z_{j,n+1}-\tilde Z_{j,n}}{\Dt}+\Frac{A_{21}}{2\Dx}\left(u_{j+1,n}-u_{j-1,n}\right)
+\Frac{A_{22}}{2\Dx}\left(\tilde Z_{j+1,n}-\tilde Z_{j-1,n}\right)\\
\\
-\Frac{\bar Q_{21}}{2\Dx}\delta_x^2 \tilde u_{j,n}-\Frac{\bar Q_{22}}{2\Dx}\delta_x^2 \tilde Z_{j,n}=\Sum_{l=-1,0,1}\Big(\widetilde{\bar\Bcal_{l}}(u_{j+l,n})-\bar b^{21}_l u_{j+l,n}-\bar b^{22}_l \tilde Z_{j+l,n}\Big).
\end{array}
\end{equation}
where, for $l=-1,0,1$, vector $\bar \Bcal_l(u) = (\bar\Bcal_l^1(u),\widetilde{\bar\Bcal_l}(u))^T$, follows the notation given in \eqref{vectDecDim} and the block decomposition of matrices $\bar b_l$ and  $\bar Q$ follows the definition given in \eqref{blockDim}.

\subsection{Decay properties of the local truncation error}\label{sec_decayTronc_gen}
In this section we  focus on the local truncation error for the general scheme \eqref{numcd}.
By applying the time decay properties given in Section \ref{decays}, we will show how it is possible to 
build up numerical schemes which are more accurate for large times.

Set, for $i=1,...,m$, 
\begin{equation}\label{scheme_param}
\begin{array}{l}
C = diag(C^i), \quad \bar C = D C D^{-1}, \quad \Ccal(u)= (\Ccal_i(u))^T, \quad \gamma^i = (\beta^i_1 -\beta^i_{-1}),\\
\\
\bar G = (\bar b_1 - \bar b_{-1}), \quad \Gamma = (\Bcal_1(u)-\Bcal_{-1}(u)), \quad \bar \Gamma = (\bar\Bcal_1(u)-\bar\Bcal_{-1}(u)).
\end{array}
\end{equation}
For $i=1,..,m$, we have
$$
\begin{array}{l}
\bar \Bcal_l(u_{j+l,n})=\bar \Bcal_l(u_{j,n})+l\, \Dx\, \bar\Bcal_l'(u_{j,n}) \partial_x u (x_j,t_n)+O(\Dx^2), 
\end{array}
$$
where $ \bar\Bcal_{l}'\in\R^{km\times k} $ is the Jacobian matrix of $\bar \Bcal_l$. Using the Taylor expansion and the consistency property \eqref{consistency},
the local truncation error for the scheme \eqref{numcd_uz} becomes
\begin{equation}\label{Tu_gen}
\begin{array}{ll}
\Tcal_\ub =& \Frac{\Dt}{2} \partial_{tt} \ub - \Dx \Frac{\bar Q_{11}}{2}\partial_{xx} \ub - \Dx \Frac{\bar Q_{12}}{2}\partial_{xx} \tilde\Zb\\ \\
&+ \Dx \Big[ T^u_0(\ub) + T^u_1 \partial_x \ub + S^u_0 \tilde \Zb +S^u_1 \partial_x\tilde Z\Big] + O(\Dt^2 +\Dx^2),
\end{array}
\end{equation}
and
\begin{equation}\label{Tz_gen}
\begin{array}{ll}
\Tcal_\zb= &\Frac{\Dt}{2} \partial_{tt}\tilde Z -\Dx \Frac{\bar Q_{21}}{2}\partial_{xx} \ub - \Dx \Frac{\bar Q_{22}}{2}\partial_{xx} \tilde\Zb\\ \\
&+\Dx \Big[ T^z_0(\ub) + T^z_1\partial_x\ub+S^z_0\tilde\Zb+S^z_1\partial_x\tilde\Zb\Big]+ O(\Dt^2 +\Dx^2),
\end{array}
\end{equation}
where
\begin{equation}
\begin{array}{cc}
T^u_0(\ub)= -\left(D_{11}\Ccal^1(\ub)+D_{12}\tilde\Ccal(\ub)-\bar C_{11} \ub\right) \in\R^{k}, &

T^u_1 = -\bar\Gamma_1' + \bar G_{11}\in \R^{k\times k},\\
\\
T^z_0(\ub) = -(D_{21}\Ccal^1(\ub)+D_{22}\tilde\Ccal(\ub)-\bar C_{21} \ub)\in\R^{m_2},
&
T^z_1 = (- \widetilde{\bar\Gamma}' +\bar G_{21})\in\R^{m_2\times k},\\
\\
S^u_0 = \bar C_{12}\in\R^{k\times m_2},
&
S^u_1 = \bar G_{12}\in\R^{k\times m_2},\\
\\

S^z_0 = \bar C_{22}\in\R^{m_2\times m_2},
&
S^z_1 = \bar G_{22}\in\R^{m_2\times m_2},
\end{array}
\end{equation}
where $\widetilde{\bar{\Gamma}}'$ is the $ m_2\times k$ jacobian matrix of $\tilde{\bar\Gamma}$. Clearly,  the scheme  \eqref{numcd} is at least consistent, which means it is formally of order $O(\Dx+\Dt)$. Taking $\Gamma_i=0$, $\gamma^i=0$, $C^i=0$ and $\Ccal_i=0$, for $i=1=,...,m$, we get the standard upwind scheme with the pointwise approximation for the source term and the local truncation error is just given by
$$\begin{array}{l}
\Tcal_\ub(x,t)=\Frac{\Dt}{2} \partial_{tt} \ub - \Dx \Frac{\bar Q_{11}}{2}\partial_{xx} \ub - \Dx \Frac{\bar Q_{12}}{2}\partial_{xx} \tilde\Zb + O(\Dx^2 +\Dt^2), \\
\\ 
\Tcal_\zb(x,t)=\Frac{\Dt}{2} \partial_{tt}\tilde Z -\Dx \Frac{\bar Q_{21}}{2}\partial_{xx} \ub - \Dx \Frac{\bar Q_{22}}{2}\partial_{xx} \tilde\Zb+ O(\Dx^2 + \Dt^2).
\end{array}$$

Using the time decay estimates in Theorems \ref{teo_decay_u} and \ref{decay_t}, we obtain for a general approximation the following estimates for the local truncation error, as $t\rightarrow +\infty$,
$$\begin{array}{cc}
\Tcal_\ub(x,t)=   O(\Dx\ t^{-3/2}) +  O(\Dt \ t^{-3/2}), & 
\Tcal_\zb(x,t)= O(\Dx\ t^{-3/2})+  O(\Dt \ t^{-3/2}).
\end{array}$$

Starting from the general scheme \eqref{numcd}, we would like to improve the decay property of this local truncation error to build up more accurate numerical schemes. 
The main idea is to chose the free parameters of the scheme to delete  the terms that decay more slowly in \eqref{Tu_gen}-\eqref{Tz_gen}, i.e. the terms which decays as $t^{-3/2}$.

Let   $g_i=diag(\gamma_{(i-1)m_1+1},...,\gamma_{i m_1})$ for $i=1,...,m$ and $G=diag(g_1,...,g_m)$.

\begin{teo}[Local Truncation Error]\label{prop_aho_gen}
Let $\Dt/\Dx = \rho$ be fixed and let $H=diag(h_1,\dots,h_m)$ be the block diagonal matrix given in \eqref{matrixH}. Recall that, by \eqref{S}, $ A_{11} = \sum_{i=1}^{m} \lambda_i h_i^{-1}$, and set 
$\quad P=\sum_{i=1}^{m} \lambda_i^2 h_i^{-1}.$
Assume $A_{11}\ne 0$ and that the following condition holds:
\begin{equation}\label{invert}
\hbox{the matrix } (\lambda_i I_k-A_{11}) \hbox{ is invertible for }i=1,...,m.
\end{equation}
If  we make the following choice for the coefficients of the scheme \eqref{numdiag}, 
\begin{equation}\label{coeffs_C}
 C = -\Frac{\rho}{2}I_{km}, \quad \Ccal = C M(u) = -\Frac{\rho}{2} M(u),
\end{equation}
\begin{equation}\label{coeffs_G}
g_i = -\left(\Frac{1}{2} q_i h_i^{-1} -\Frac{\rho}{2}h_i^{-1}\left(P-(\lambda_i I_k -A_{11})^2\right)\right)(\lambda_i I_k-A_{11})^{-1}h_i
\end{equation}
and
\begin{equation}\label{coeffs_Gamma}
\Gamma_i'(u) =  g_i M'_i(u) +\Frac{\rho}{2} \left((h_i^{-1}- M_i'(u))A_{11}+\lambda_i M_i'(u)-h_i^{-1}\Sum_{j=1}^m \lambda_j M'_j(u)\right),
\end{equation}
both for $i=1,...,m$, then the local truncation error of the scheme \eqref{numdiag} decays as 
\begin{equation}\label{TAHO_TD_gen}
\begin{array}{cc}
\Tcal_u(x,t)=   O\left( \Dx\ t^{-2}\right)+O\left( \Dx^2\ t^{-3/2}\right), & 
\Tcal_z(x,t)= O\left(\Dx\ t^{-2}\right)+O\left( \Dx^2\ t^{-3/2}\right).
\end{array}
\end{equation}
\end{teo}
\begin{proof}
Set
\begin{equation}\label{tildeTerms}
\begin{array}{l}
\tilde T_0 = -\tilde Q(\ub),\\
\tilde T_1 = -(A_{22}(\tilde Q'(\ub))+(\tilde Q'(\ub)) A_{11}-A_{21}),\
\tilde T_2 = A_{21}A_{11}+A_{22}A_{21}, \\
\tilde S_1 = 2 A_{22} -(\tilde Q'(\ub))A_{12}, \ 
\tilde S_2 = A_{21}A_{12}+A_{22}^2.
\end{array}
\end{equation} 
Replacing these expressions in \eqref{Tu_gen}-\eqref{Tz_gen}, and using the structure of the system, yields 
\begin{equation}\label{TuTz_proof}
\begin{array}{ll}
\Tcal_u &= \Dx\Big[ T^u_0(\ub) + \left(T^u_1+S^u_1(\tilde Q'(\ub))\right)\, \partial_x \ub + \left(\Frac{\rho}{2}A_{11}^2 -\Frac{\bar Q_{11}}{2} - S^u_1 A_{21}\right)\, \partial_{xx} \ub+ S^u_0\, \tilde\Zb\\
\\
& - \left(\Frac{\rho}{2}A_{12}+S^u_1\right)\, \partial_{tx} \tilde\Zb + \left(-\Frac{\bar Q_{12}}{2}+\Frac{\rho}{2}A_{11}A_{22}-S^u_1 A_{22}\right)\, \partial_{xx} \tilde\Zb\Big] +O(\Dx^2\ t^{-3/2}),\\
\\
\Tcal_z&= \Dx \Big[ \left(\Frac{\rho}{2} \tilde T_0(\ub) +T^z_0(\ub)\right) + \left(T^z_1+\Frac{\rho}{2}\tilde T_1 +(S^z_1+\Frac{\rho}{2}\tilde S_1)(\tilde Q'(\ub))\right)\, \partial_x \ub \\
\\
&+ \left(\Frac{\rho}{2}\tilde T_2 -\Frac{\bar Q_{21}}{2}-(S^z_1+\Frac{\rho}{2}\tilde S_1)A_{21}\right)\, \partial_{xx} \ub+ \left(\Frac{\rho}{2}\, I_{m_2} + S^z_0\right)\, \tilde\Zb - \left(S^z_1+\Frac{\rho}{2}\tilde S_1\right)\, \partial_{tx} \tilde\Zb \\
\\
&+ \left(\Frac{\rho}{2}\tilde S_2-\Frac{\bar Q_{22}}{2}-(S^z_1+\frac{\rho}{2}\tilde S_1)A_{22}\right)\, \partial_{xx} \tilde\Zb\Big] +O(\Dx^2\ t^{-3/2} ),
\end{array}
\end{equation}
Our choice of the coefficients $\Bcal_l^i(u)$, $\beta_l^i$, $C^i$ and $\Ccal_i(u)$ is made just to cancel the coefficients of the slowly terms, i.e. 
\begin{eqnarray}\label{coeffs_1}
T^\ub_0 = 0,  & S^u_0=0, \label{c1-c2}\\
\Frac{\rho}{2} \tilde T_0(\ub) +T^z_0(\ub) = 0, & \Frac{\rho}{2}\, I_{m_2} + S^z_0 =0, \label{c5-c8}\\
T^u_1+S^u_1\tilde Q'(\ub)=0, & \Frac{\rho}{2}A_{11}^2-\Frac{\bar Q_{11}}{2} - S^u_1 A_{21} = 0,\label{c3-c4}\\
T^z_1+\Frac{\rho}{2}\tilde T_1 +(S^z_1+\Frac{\rho}{2}\tilde S_1)\tilde Q'(\ub)=0, & \Frac{\rho}{2}\tilde T_2-\Frac{\bar Q_{21}}{2}-(S^z_1+\Frac{\rho}{2}\tilde S_1)A_{21}=0.\label{c6-c7}
\end{eqnarray}
Therefore, the local truncation error reduces to
\begin{equation}
\begin{array}{ll}
\Tcal_u =& \Dx\Big[ - \left(\Frac{\rho}{2}A_{12}+S^u_1\right)\, \partial_{tx} \tilde\Zb \\ \\&  + \left(-\Frac{\bar Q_{12}}{2}+\Frac{\rho}{2}A_{11}A_{22}-S^u_1 A_{22}\right)\, \partial_{xx} \tilde\Zb\Big] +O(\Dx^2\ t^{-3/2}),\\
\\
\Tcal_z= &\Dx \Big[ - \left(S^z_1+\Frac{\rho}{2}\tilde S_1\right)\, \partial_{tx} \tilde\Zb \\ \\ &+ \left(\Frac{\rho}{2}\tilde S_2-\Frac{\bar Q_{22}}{2}-(S^z_1+\frac{\rho}{2}\tilde S_1)A_{22}\right)\, \partial_{xx} \tilde\Zb\Big] +O(\Dx^2\ t^{-3/2}).
\end{array}
\end{equation}
and by the estimates given in \eqref{decay_u}-\eqref{decay_zt} the thesis is achieved.

We need now to show that system \eqref{c1-c2}-\eqref{c6-c7}  has always a solution given by the relations \eqref{coeffs_C}-\eqref{coeffs_Gamma}.
The terms given in \eqref{coeffs_C} for the $\Ccal$ and $C$ coefficients are obtained from \eqref{c1-c2} and \eqref{c5-c8}. Indeed, we get 
\begin{eqnarray*}
\left(-D_{11}\Ccal^1(\ub)-D_{12}\tilde\Ccal(\ub)+\bar C_{11} \ub\right)=0, & \bar C_{12}=0, \\
\left(-\Frac{\rho}{2} \tilde Q(\ub)  +(- D_{21}\Ccal^1(\ub)-D_{22}\tilde\Ccal(\ub)+\bar C_{21} \ub)\right) = 0, & \Frac{\rho}{2}\ I_{m_2} +\bar C_{22} =0,
\end{eqnarray*}
or in  a more compact form,
\begin{eqnarray*}
-\left(\begin{array}{cc}
D_{11} & D_{12}\\
D_{21} & D_{22}
\end{array}\right)
\left(\begin{array}{c}
\Ccal^1(\ub) \\ \tilde\Ccal(\ub)
\end{array}\right) + 
\left(\begin{array}{cc}
\bar C_{11} & \bar C_{12}\\
\bar C_{21} & \bar C_{22}
\end{array}\right) 
\left(\begin{array}{c}
\ub \\ \tilde Z
\end{array}\right) = 
\left(\begin{array}{c}
0 \\ \Frac{\rho}{2} \left(\tilde Q(\ub)-I_{m_2}\tilde Z\right)
\end{array}\right),
\end{eqnarray*}
which can be rewritten as
$$ -D \Ccal + D C D^{-1} Z = \left(\begin{array}{c}
0 \\ \Frac{\rho}{2} \left(\tilde Q(\ub)-\tilde Z\right)
\end{array}\right). $$
Now, we multiply on the left  by the matrix $D^{-1}$ to obtain
\begin{equation}\label{cglob}
-\Ccal + C f = \Frac{\rho}{2} D^{-1}  \left(\begin{array}{c}
0 \\ \tilde Q(\ub)-\tilde Z
\end{array}\right) = 
\Frac{\rho}{2} \left(-f + M(u)\right),
\end{equation}
that gives  \eqref{coeffs_C}.

We shall now focus on the derivation of the relations \eqref{coeffs_G} and \eqref{coeffs_Gamma}. From \eqref{c3-c4}-\eqref{c6-c7}, for the matrix of free coefficients $\bar G$ we have to impose
\begin{equation}\label{Gpassaggio}
\bar G_{12} A_{21} = -\Frac{\bar Q_{11}}{2} + \Frac{\rho}{2}A_{11}^2, \
\bar G_{22} A_{21} = -\Frac{\bar Q_{21}}{2} + \Frac{\rho}{2}\left(\tilde T_2 -\tilde S_1A_{21}\right).
\end{equation}

Now, using the notation \eqref{blockDiagDim} for the two diagonal matrices $G$ and $Q$ and  
 by applying relations \eqref{matrixH} and \eqref{Dinv}, we have that for $\bar G=D G D^{-1}$ and $\bar Q = D Q D^{-1}$
\begin{eqnarray*}
\bar G_{12}= G_1 H_1^{-1}D_{21}^T + D_{12}G_2H_2^{-1}D_{22}^T, \\
\bar G_{22}= D_{21}G_1H_1^{-1}D_{21}^T+D_{22}G_2H_2^{-1}D_{22}^T,\\
\bar Q_{11} = Q_1 H_1^{-1} +D_{12}Q_2 H_2^{-1}D_{12}^T,\\
\bar Q_{21} = D_{21}Q_1H_1^{-1} + D_{22}Q_2  H_2^{-1}D_{12}^T.
\end{eqnarray*} 
Since $D_{11}=I_k$, we have that \eqref{Gpassaggio} becomes
\begin{equation}\label{G1}
\left(
\begin{array}{c}
G_{1}  H_1^{-1}D_{21}^T A_{21}\\ 
G_{2} H_2^{-1}D_{22}^T A_{21}
\end{array}
\right) = 
-\Frac{1}{2}
\left(\begin{array}{c}
Q_{1}H_1^{-1} \\
Q_{2} H_2^{-1}D_{12}^T 
\end{array}
\right)
+ \Frac{\rho}{2} D^{-1}
\left(\begin{array}{c}
A_{11}^2 \\
A_{21}A_{11} - A_{22}A_{21}
\end{array}
\right),
\end{equation}
where we used relations \eqref{tildeTerms} to sort out the term $\tilde T_2 -\tilde S_1A_{21} = A_{21}A_{11} - A_{22}A_{21}$. Here we used the fact that, neglecting the terms with a faster decay, we can replace  $\tilde S_1$ with $2A_{22}$. Actually, since the term $\tilde S_1$ in \eqref{TuTz_proof} multiplies only terms that decay faster than $t^{-3/2}$, it is possible to write 
$$ \tilde S_1 = 2 A_{22} -\tilde Q'(u) A_{12}= 2 A_{22} + O(\Delta x t^{-2}).$$
From \eqref{G1} we get
\begin{equation}\label{G2}
\begin{array}{l}
G_{1}H_1^{-1}  D_{21}^T A_{21} = -\Frac{1}{2} Q_{1}H_1^{-1} + \Frac{\rho}{2} H_1^{-1}\left(A_{11}^2+ D_{21}^T\left(A_{21}A_{11} - A_{22}A_{21}\right)\right)\\ 
\\
G_{2} H_2^{-1} D_{22}^T A_{21} = -\Frac{1}{2}Q_{2} H_2^{-1} D_{12}^T + \Frac{\rho}{2} H_2^{-1} \left(D_{12}^TA_{11}^2+ D_{22}^T\left(A_{21}A_{11} - A_{22}A_{21}\right)\right),
\end{array}
\end{equation}

Using  the specific form of $D$ and relations \eqref{Drelaz}, by algebraic considerations we get
$$  A_{11}^2+ D_{21}^T\left(A_{21}A_{11} - A_{22}A_{21}\right)= P-(\lambda_1 I_k -A_{11})^2, $$
and
$$ D_{12}^TA_{11}^2+ D_{22}^T\left(A_{21}A_{11} - A_{22}A_{21}\right)= D^T_{12}P-D^T_{12}A_{11}^2-\Lambda_2^2D^T_{12}+2\Lambda_2D^T_{12}A_{11}.$$

Then, we get for $i=1,..,m$ 
\begin{equation}
g_i h_i^{-1}(\lambda_i I_k -A_{11}) = -\Frac{1}{2} q_i h_i^{-1} +\Frac{\rho}{2}h_i^{-1}\left(P-(\lambda_i I_k -A_{11})^2\right).
\end{equation}
Assuming for $i=1,...,m$, $(\lambda_i I_k -A_{11})$ to be invertible, we obtain relations \eqref{coeffs_G}.
 
Finally, we need to compute the vector function $\Gamma(u)$. From \eqref{c6-c7} we obtain the two following relations
\begin{equation}\label{GG1}
-\bar\Gamma'_1+\bar G_{11}+\bar G_{12}\tilde Q'(\ub)=0,  \
\\
-\tilde{\bar\Gamma}'+\bar G_{21} + \bar G_{22} \tilde Q'(\ub) = -\frac{\rho}{2}\left(\tilde T_1 +\tilde S_1\tilde Q'(\ub)\right).
\end{equation}
Since
$$ M(u) = D^{-1}\left(\begin{array}{c} u \\ \tilde Q(u) \end{array}\right) \Rightarrow M'(u) = D^{-1}\left(\begin{array}{c} I_k \\ \tilde Q'(u) \end{array}\right),$$
equations \eqref{GG1} reduce to 
\begin{equation}\label{GG2}
-\Gamma'(u) + G M'(u) = -\Frac{\rho}{2} D^{-1} \left(\begin{array}{c} 0 \\  \tilde T_1 +\tilde S_1\tilde Q'(\ub)\end{array}\right),
\end{equation}
where $\tilde T_1 +\tilde S_1\tilde Q'(\ub)=A_{22}\tilde Q'(u)-\tilde Q'(u) A_{11}+A_{21}$. 
Since $ \tilde Q'(\ub) = D_{21}M'_1(u)+D_{22}\tilde M'(u)$  and  $M'_1(u)+D_{12}\tilde M'(u)=I_k$,
by using relations \eqref{Drelaz}, we obtain that the right side of \eqref{GG1} is equal to
$$
 -\Frac{\rho}{2} \left(\begin{array}{c} H_1^{-1} \left((I_k-H_1 M'_1(u))A_{11}+\Lambda_1H_1M'_1(u)-\Sum_{i=1}^m \lambda_i M'_i(u)\right) \\  H_2^{-1}\left(D^T_{12}A_{11} - H_2 \tilde M'(u)A_{11}+\Lambda_2 H_2 \tilde M'(u)-D^T_{12}\Sum_{i=1}^m \lambda_i M'_i(u) \right)\end{array}\right).
$$
Therefore, we get for every vector $\Gamma_i(u)\in\R^k$, $i=1,...,m$, that 
\begin{equation}
-\Gamma_i'(u) + g_i M'_i(u) = -\Frac{\rho}{2} \left((h_i^{-1}- M_i'(u))A_{11}+\lambda_i M_i'(u)-h_i^{-1}\Sum_{i=1}^m \lambda_i M'_i(u)\right).
\end{equation}
\end{proof}

\begin{remark}\label{remarkSchemeA0}
As previously observed in Remark \ref{remarkA0}, when $A_{11}=0$, we have, for the second-order time derivatives,
$$ \partial_{tt} \ub \sim t^{-2}, \quad \partial_{tt} \tilde Z  \sim t^{-2}.$$
Therefore, in dealing with \eqref{Tu_gen}-\eqref{Tz_gen}, we do not need to delete the second-order time derivatives $\ub_{tt}$ and $\tilde Z_{tt}$, and relations \eqref{coeffs_C}-\eqref{coeffs_Gamma} reduces to
\begin{equation}\label{coeffs_C_A0}
 C = 0, \quad \Ccal = 0.
\end{equation}
Besides, for each $i=1,...,m$ such that $\lambda_i\ne 0$, we can choose
\begin{equation}\label{coeffs_GGamma_A0}
g_i=\beta^i_1 -\beta^i_{-1} = -\Frac{1}{2\lambda_i} q_i, \quad   \Gamma_i(u)=\Bcal^i_{1}-\Bcal^i_{-1} =  g_i M_i(u).
\end{equation}
\end{remark}
Then,  we can select  $q_i=|\lambda_i|$,for $i=1,...,m$, and we are free to choose
\begin{equation}
\beta^i_0 = \Frac{1}{2}, \quad \Bcal^i_0(u) = \Frac{M_i(u)}{2}.
\end{equation}
Therefore, we obtain  an upwind scheme for  system \eqref{sysdiag}, with the classical Roe upwinding approximation for the source term \cite{Roe}, namely
\begin{equation}\label{roe_gen}
\begin{array}{ll}
\Frac{f^i_{j,n+1}-f^i_{j,n}}{\Dt}+\Frac{\lambda_i}{2\Dx}\left(f^i_{j+1,n}-f^i_{j-1,n}\right)-\Frac{|\lambda_i|}{2\Dx}\delta_x^2 f^i_{j,n}\\ \\ \qquad =\Frac{M_i(u^n_{j-1})+2M_i(u^n_j)+M_i(u^n_{j+1})}{4}
+\frac{sgn(\lambda_i)}{4}(M_i(u^n_{j-1})-M_i(u^n_{j+1}))\\ \\  \qquad\  -\Frac{f^i_{j-1,n}+2f^i_{j,n}+f^i_{j+1,n}}{4}- \frac{sgn(\lambda_i)}{4}(f^i_{j-1,n}-f^i_{j+1,n}),
\end{array}
\end{equation} 
From now on, we shall refer to scheme \eqref{roe_gen} as the ROE scheme. 
Then, we have  just proved the following result.
\begin{prop} For $A_{11}=0$, the local truncation error of the ROE scheme \eqref{roe_gen} verifies the time asymptotic estimate \eqref{TAHO_TD_gen}.
\end{prop}

\section{A monotone Time-AHO scheme for the $2\times 2$ case}\label{sec_monot2x2}
In this section we  apply the general result stated in Theorem \ref{prop_aho_gen} for the $2\times 2$ case described in the Example \eqref{example2x2}.
For the artificial viscosity, we  assume $q_1=q_2=q$, $q\in\R^+$.
\\
Hence, by applying the results given in Theorem \ref{prop_aho_gen}, we get for the scheme parameters the following expressions:
\begin{equation}\label{Cc}
 C_1=C_2=-\Frac{\rho}{2}, \quad\quad  \Ccal_1(u) = -\Frac{\rho}{2}M_1(u), \quad\quad \Ccal_2(u) = -\Frac{\rho}{2}M_2(u)\end{equation}
\begin{equation}\label{gamma}
 \gamma_1= \Frac{q}{2 a_+}+\Frac{a\rho}{2 a_+}\left(2\lambda + a\right),\quad\quad \gamma_2= -\Frac{q}{2 a_-}+\Frac{a\rho}{2 a_-}\left(2\lambda-a\right), 
\end{equation}
\begin{equation}\label{Gamma}
 \Gamma_1(u) = \Frac{q-a^2\rho}{2 a_+} M_1(u)-\Frac{a_-^2}{4\lambda}\rho u, \quad\quad
\Gamma_2(u) = -\Frac{q-a^2\rho}{2 a_-} M_2(u)+\Frac{a_+^2}{4\lambda}\rho u.
 \end{equation}
To complete the definition of scheme \eqref{numcd} it is still 
necessary to choose four more free parameters, such as $\Bcal^{1,2}_0(\cdot)$ and $\beta^{1,2}_0$. For the $2\times 2$ case such parameters can be defined by applying the monotonicity conditions.

Starting from the scheme written in its diagonal form \eqref{numdiag}, the monotonicity conditions are given by the following relations, see \cite{DMR}:
\begin{equation}\label{monotone}
\begin{array}{l}  \Bcal_{1;2,l}'(\cdot)\geq 0 \quad \forall \, l=-1,0,1,\\
1-\rho q+\Dt\left(  \Bcal_{1;2,0}'(u) -\beta^{1;2}_0\right)\geq 0,\\
 \left\{ \begin{array}{l}
\Frac{\rho \lambda_{1;2}}{2}+\Frac{\rho
  q}{2}+\Dt \left(\Bcal_{{1;2},-1}'(u) -\beta^{1;2}_{-1}\right)\geq 0,\\ \\
-\Frac{\rho \lambda_{1;2}}{2}+\Frac{\rho q}{2}+\Dt \left( \Bcal_{{1;2},1}'(u) -\beta^{1;2}_{1}\right)\geq 0.
\end{array} \right.
\end{array} \end{equation}

\begin{prop}[Monotonicity]\label{prop_monot}
Assume $a>0$ and $\lambda \geq \|F'(u)\|_\infty$. Under the assumptions of Theorem \ref{prop_aho_gen}, the scheme \eqref{numcd} for the $2\times 2$ case  verifies the monotonicity conditions \eqref{monotone} for the choices:
\begin{equation}\label{B0}
\begin{array}{ll}
\Bcal^1_{0}(u)= M_1(u) -|\Gamma_1(u)| + \Dx\ C^1(u), &
\Bcal^2_{0}(u)= M_2(u) -|\Gamma_2(u)| + \Dx\ C^2(u).
\end{array}
\end{equation}
\begin{equation}\label{b0}
\begin{array}{ll}
\beta^1_0=1-\gamma_1 + \Dx\ c_1, &
\beta^2_0=1+\gamma_2 + \Dx\ c_2, 
\end{array}\end{equation}
under the CFL conditions
\begin{equation}\label{cfl_Dt}
  \Dt\leq \min{\left(\Frac{1-\lambda\rho}{1+\gamma_1},\Frac{1-\lambda\rho}{1-\gamma_2}\right)}, 
\end{equation}
and, for $\lambda > 2 a$
\begin{equation}\label{cfl}
\begin{array}{ll}
\Dx\leq \Frac{a^2}{\lambda+a}, & 
\rho \leq \min{\left(\Frac{1}{\lambda},\Frac{2\lambda^2 m_-}{2a^2\lambda m_-+a_+a_-^2},\Frac{2\lambda^2 m_+}{2a^2\lambda m_++a_+^2a_-}\right)},
\end{array}
\end{equation}
where $ m_{1;2} = \min_{u}(M'_{1;2}(u))> 0 $.
Otherwise, if $a< \lambda < 2 a$, we get the supplementary requirement
$$ \rho \geq \Frac{|\lambda-2 a|}{a^2-\Dx\ a_-}. $$
\end{prop}
\begin{proof}
From consistency \eqref{consistency}, we write
$$
\begin{array}{l}
\Bcal^{1;2}_{-1}(u)=\Frac{1}{2}\Big(M_{1;2}(u)-\Bcal^{1;2}_{0}(u)-\Gamma_{1;2}(u)+\Dx\ C_{1;2}(u)\Big), \\ \\
\Bcal^{1;2}_{1}(u)=\Frac{1}{2}\Big(M_{1;2}(u)-\Bcal^{1;2}_{0}(u)+\Gamma_{1;2}(u)+ \Dx\ C_{1;2}(u)\Big),\\
\\
\beta^{1;2}_{-1}=\Frac{1}{2}\Big(1-\beta^{1;2}_{0}-\gamma_{1;2}+\Dx\ c_{1;2}\Big), \ 
\beta^{1;2}_{1}=\Frac{1}{2}\Big(1-\beta^{1;2}_{0}+\gamma_{1;2} + \Dx\ c_{1;2}\Big).
\end{array}
$$
The first condition in \eqref{monotone} is equivalent to 
\begin{equation}\label{monot_1_1}
 M'_{1;2}-|\Gamma'_{1;2}| + \Dx\ C'_{1;2} \geq \Bcal'_{0,{1;2}}\geq 0,
\end{equation}
and the third condition in \eqref{monotone} is equivalent to
\begin{equation}\label{monot_3_1}
\left\{\begin{array}{l}
 \Frac{\rho}{2}(\pm \lambda+q)-\Frac{\Dt}{2}\Big(1-\beta^2_{0}\mp\gamma_2 + \Dx\ c_+\Big)\geq 0,\\
\\
 \Frac{\rho}{2}(\mp \lambda+q)-\Frac{\Dt}{2}\Big(1-\beta^1_{0}\mp\gamma_1 + \Dx\ c_- \Big)\geq 0.
\end{array}\right.
\end{equation}

It is natural to assume that the CFL ratio $\rho$ verifies the standard hyperbolic condition 
\begin{equation}\label{hyp_cfl}
\rho \leq \Frac{1}{\lambda}.
\end{equation}
Then, from relations \eqref{gamma}, we have $\gamma^+\leq 0$ and $\gamma^-\geq 0$ and 
for $q\geq\lambda$, we get monotonicity by choosing in \eqref{monot_3_1},
\begin{equation}\label{monotb0}\begin{array}{ll}
\beta^1_0=1-\gamma_1 + \Dx\ c_-, &
\beta^2_0=1+\gamma_2 + \Dx\ c_+, 
\end{array}\end{equation}
under the limitation required in the second condition in  \eqref{monotone}, 
\begin{equation}
  \Dt\leq \min{\left(\Frac{1-\lambda\rho}{1+\gamma_1},\Frac{1-\lambda\rho}{1-\gamma_2}\right)}.
\end{equation}
To verify the request \eqref{monot_1_1},  we set first
\begin{equation}\label{monotB0}
\begin{array}{ll}
\Bcal'_{1,0}(u)= M'_1 -|\Gamma'_1| + \Dx\ C'_1, &
\Bcal'_{2,0}(u)= M'_2 -|\Gamma'_2| + \Dx\ C'_2.
\end{array}\end{equation}
For 
$$
\rho \leq \min{\left(\Frac{2\lambda^2 m_-}{2a^2\lambda m_-+a_+a_-^2},\Frac{2\lambda^2 m_+}{2a^2\lambda m_++a_+^2a_-}\right)}
$$
we get 
$$ \Gamma'_2 \leq 0, \quad \Gamma'_1\geq 0,$$
then conditions \eqref{monot_1_1} become
$$\begin{array}{l} 
\left(1-\Frac{\lambda}{2a_-}+\Frac{\rho}{2}\left(\frac{a^2}{a_-}-\Dx\right)\right)M'_2 +\Frac{a^2_+}{4\lambda}\rho\geq 0, \\
\\
\left(1-\Frac{\lambda}{2a_+}+\Frac{\rho}{2}\left(\frac{a^2}{a_+}-\Dx\right)\right)M'_1 +\Frac{a^2_-}{4\lambda}\rho\geq 0,
\end{array}
$$
that are verified under the following limitations, 
\begin{equation}\label{monot_noLin}
\lambda > \max{\left(\|F'(u)\|_\infty, 2 F'(0)\right)}, 
\quad \Dx\leq \Frac{a^2}{\lambda+a}.
\end{equation}
If we choose 
$$ a<\lambda<2 a, $$
we get a limitation from the bottom for $\rho$,
$$ \rho \geq \Frac{|\lambda-2 a|}{a^2-\Dx\ a_-}. $$
\end{proof}

The monotone Time-AHO scheme has then the following expression,
\begin{equation}\label{aho}
\begin{array}{l}
\Frac{f^1_{j,n+1}-f^1_{j,n}}{\Dt}-\frac{\lambda}{\Dx}\left(f^1_{j+1,n}-f^1_{j,n}\right)\\ 
\\
\begin{array}{ll}
=&\left(1-\frac{\lambda-a^2\rho}{2 a_+}\right)M_1(u_{j,n})+\frac{\lambda-a^2\rho}{2 a_+}M_1(u_{j+1,n})-\left((1-\gamma^-)f^1_{j,n}+\gamma^- f^1_{j+1,n}\right)\\
\\
&-\frac{\rho a_-^2}{4\lambda}(u_{j+1,n}-u_{j,n})+\frac{\rho\Dx}{2}\left(f^1_{j,n}-M_1(u_{j,n})\right).\\
\end{array}\\
\\
\Frac{f^2_{j,n+1}-f^2_{j,n}}{\Dt}+\frac{\lambda}{\Dx}\left(f^2_{j,n}-f^2_{j-1,n}\right)\\ 
\\
\begin{array}{ll}
=&\frac{\lambda-a^2\rho}{2 a_-}M_2(u_{j-1,n})+\left(1-\frac{\lambda-a^2\rho}{2 a_-}\right)M_2(u_{j,n})-\left(|\gamma^+| f^2_{j-1,n}+(1-|\gamma^+|)f^2_{j,n}\right)\\
\\
&+\frac{\rho a_+^2}{4\lambda}(u_{j,n}-u_{j-1,n})+\frac{\rho\Dx}{2}\left(f^2_{j,n}-M_2(u_{j,n})\right).
\end{array}
\end{array}
\end{equation}
Notice that, the approximation of the source terms involve only the values of the solutions on the "upwinding-nodes", i.e. $(x_{j},x_{j+1})$ and $(x_{j-1},x_j)$ respectively for the first and the second equations.
\\
The scheme may be then considered as an extension of the known approximation with the upwinding of the source term \eqref{roe_gen} described in Remark \ref{remarkSchemeA0}. 
Let us stress on the fact that, when $F'(0)=0$, the local truncation error of the ROE scheme \eqref{roe_gen} verifies the decay properties obtained by the Time-AHO schemes in \eqref{TAHO_TD_gen}.

\section{A monotone Time-AHO scheme for the $3\times 3$ case}\label{3x3case}
In this section we compute a TAHO approximation for the $3\times 3$ case of example \eqref{example3x3}. In particular, we study the case where $F(u)=a(u-u^2)$, with $a>0$. Then we have
$F^\prime(0)=a>0$. 

We recall that $\alpha=|a|+\beta \lambda \in ]a,\lambda[$ and 
\begin{equation}\label{mono-M}
\lambda-\alpha+|a| > |F^\prime(u)|.
\end{equation}
According to Proposition \ref{prop_aho_gen}, we have $ A_{11}=a$, and $P= \lambda \alpha$,
and the coefficients $h_i$ are
\[ h_1= \frac{2 \lambda}{\alpha -a} , \quad h_2=
\frac{\lambda}{\lambda-\alpha}  ,\quad h_3= \frac{2 \lambda}{\alpha
  +a}. \]
We choose to take $q_1=q_3=\lambda$, $q_2=0$. Consequently we find
\begin{equation}\label{gi33}
g_1=\frac{\lambda(1-\rho \alpha)}{2(\lambda+a)}+\frac{\rho(\lambda+a)}{2}
,\quad g_2=-\frac{\rho(\lambda \alpha - a^2)}{2a}, \quad g_3=-
\frac{\lambda(1-\rho \alpha)}{2(\lambda-a)}-\frac{\rho(\lambda-a)}{2}.
\end{equation}
We have
\begin{equation}\label{GAMMA-1}
\left\{ \begin{array}{l}
\Gamma_1(u)=\displaystyle \frac{\lambda(1-\rho
  \alpha)}{2(\lambda+a)}M_1(u)+\displaystyle \frac{\rho (\alpha-a)}{4
  \lambda}(au -F(u)), \\ \\
\Gamma_2(u)=- \displaystyle \frac{\rho \lambda
  \alpha}{2a}M_2(u)+\displaystyle \frac{\rho(\lambda-\alpha)}{2
  \lambda}( au -F(u)), \\ \\
\Gamma_3(u)=-\displaystyle \frac{\lambda(1-\rho
  \alpha)}{2(\lambda-a)}M_3(u)+\displaystyle \frac{\rho (\alpha+a)}{4
  \lambda}(au -F(u)) .
\end{array}
\right.
\end{equation}
For $1 \leq i \leq 3$ we have
\[ 
\beta^i_{\pm 1}=\frac{1}{2}\left( 1-\frac{\Dt}{2} \pm g_i - \beta^i_0\right),\quad \Bcal^i_{\pm 1}=\frac{1}{2}\left( M_i(u)(1-\frac{\Dt}{2}) \pm
  \Gamma_i(u) - \Bcal^i_0(u) \right),
\]
so it remains to determine the $\beta^i_0$ and $\Bcal^i_0(u)$. As for
the $2\times2$ case we use monotonicity criteria for those choices. From now
on, we use the fact that $a >0$ and we consider only $u \in [0,1]$, as
this will be satisfied in our numerical test. We have then 
\begin{eqnarray}
-a \leq F^\prime(u) \leq a, &\label{esti-Ap} \\
0<1-\displaystyle \frac{\alpha}{\lambda}\leq M^\prime_2 (u) \leq 1-\frac{\alpha-a}{\lambda}< 1,&
\label{esti-MP2} \\ 
0 < \frac{\alpha-a}{2\lambda} \leq M^\prime_i (u) \leq \frac{\alpha+a}{2\lambda} < 1, & i=1, \; i=3. \label{esti-MP}
\end{eqnarray}
Hereafter we study the monotonicity conditions for
  each of the three equations obtained by applying the general
  numerical scheme \eqref{numdiag} to this particular $3\times3$
  case. We denote 
  \[ 
  \mu=(\Dx+2a)(\lambda+a)-\lambda \alpha,
  \]
  and
  \[\nu_1=\frac{\lambda}{2(\lambda-a)}\left(1-\frac{\alpha-a}{2\lambda}\right), \quad \nu_2=- \frac{\lambda \alpha}{\lambda-a}\left(1-\frac{\alpha-a}{2\lambda}\right)+\frac{a(\alpha+a)}{\lambda}+\lambda-a.
  \]
\begin{prop}(Monotonicity).\label{mono-33}
Assume $a >0$, $\lambda> 2a$. We suppose that the following conditions
are satisfied:
\begin{eqnarray}
if \; \; \nu_2 >0  \quad then & \quad  \Dx \leq \min \left(\displaystyle \frac{\lambda \alpha}{\lambda+a},\frac{\lambda}{\nu_1+\frac{\nu_2}{\alpha}} \right),\\
if \; \; \nu_2 \leq 0  \quad then& \quad  \Dx \leq \min \left(\displaystyle \frac{\lambda \alpha}{\lambda+a},\displaystyle \frac{\lambda}{\nu_1}\right) ,
\end{eqnarray}
\begin{equation}\label{rho33}
\rho \leq \min\left(\displaystyle \frac{1}{\alpha}, \displaystyle \frac{1}{\lambda + \Dx },\displaystyle  \frac{2a  }{a \Dx + \lambda \alpha},\displaystyle  \frac{2a (\lambda-\alpha) }{a(\lambda-\alpha) \Dx +\alpha( \lambda^2+\lambda-2a^2)}
\right),
\end{equation}
\begin{equation}\label{rho331}
if \; \; \mu >0  \quad then \quad  \rho \leq \displaystyle \frac{\lambda+2a}{\mu},
\end{equation}
\begin{equation}\label{dx33}
\Dt \leq \displaystyle \frac{1}{1+\displaystyle \max_{u \in [0,1]}|\Gamma^\prime_2(u)-g_2|}.
\end{equation}
Under the assumptions of Theorem \ref{prop_aho_gen}, scheme
\eqref{numcd} for the considered $3\times 3$ case is monotone for the choices:
\begin{equation}
\beta^1_{0}=1-\frac{\Dt}{2}-g_1, \quad \beta^2_0=1-\frac{\Dt}{2}+g_2,  \quad \beta^3_{0}=1-\frac{\Dt}{2}+g_3,
\end{equation}
\begin{eqnarray}
\Bcal_0^1(u)=M_1(u)(1-\frac{\Dt}{2})-\Gamma_1(u), \quad \Bcal_0^3(u)=
M_3(u)(1-\frac{\Dt}{2})+\Gamma_3(u), 
\end{eqnarray}
and $\Bcal^2_{0}$ is a continuous function such that
 \begin{equation}\label{mono332}
  -(1-M_2^\prime(u))(1-\frac{\Dt}{2})- |\Gamma_2^\prime(u)-g_2| = (\Bcal^2_{0})^\prime(u)-\beta^2_{0}.
 \end{equation}
 Such a function exists.
\end{prop}
\begin{proof}
We do not detail the proof for equations 1 and 3, as it follows the lines of the $2\times 2$ case. Actually, under the above assumptions we have 
\[ g_1 >0, \quad \Gamma^\prime_1(u) \geq 0, \quad g_3<0, \quad \Gamma^\prime_3(u)\leq 0. \]
Then we obtain a monotone TAHO scheme on those equations.

The treatment of the second equation is quite different. We first note
that $g_2<0$ but the sign of $\Gamma^\prime_2(u)$ does not depend on
the parameters of the discretization. The monotonicity conditions are:
\begin{eqnarray}
(\Bcal ^2_{l})^\prime(u) \geq 0,& l= -1, 0 , 1,  \label{mono-20} \\
- \beta^2_{\pm 1 } +  (\Bcal^2_{ \pm 1})^\prime(u) \geq 0,& \label{mono-21} \\ 
1- \Dt (\beta^2_{0} -(\Bcal^2_{0})^\prime(u) ) \geq
0.&  \label{mono-22} 
\end{eqnarray}
The inequality (\ref{mono-21}) can be written as
\begin{equation}
-(1-M_2^\prime(u))(1-\frac{\Dt}{2})- |\Gamma_2^\prime(u)-g_2| \geq (\Bcal^2_{0})^\prime(u)-\beta^2_{0},
\end{equation}
which is implied by equality (\ref{mono332}).
In the case under consideration, it is straightforward to determine the sign of $\Gamma^\prime_2(u)-g_2$ with respect to $u$. In our numerical experiment for example, we have $u_0 \in ]0, 0.5[$ and $u_1 \in ]0.5,1[$ such that 
\begin{eqnarray}
\Gamma^\prime_2(u)-g_2 > 0 &\quad {\rm in } \quad &[0,u_0[ \cup ]u_1,1], \nonumber \\
 \Gamma^\prime_2(u)-g_2 < 0 &\quad {\rm in } \quad &]u_0,u_1[ . \nonumber 
 \end{eqnarray}
We can therefore construct a continuous function $\Bcal^2_{0}$ such that
\begin{eqnarray}
( \Bcal^2_{0})^\prime(u)=M_2^\prime(u)(1-\frac{\Dt}{2}) - \Gamma^\prime_2(u)+2g_2 & \quad {\rm when  } \quad &   \Gamma^\prime_2(u)-g_2 > 0,\nonumber \\
( \Bcal^2_{0})^\prime(u)=M_2^\prime(u)(1-\frac{\Dt}{2}) + \Gamma^\prime_2(u) & \quad {\rm when } \quad &\Gamma^\prime_2(u)-g_2 <0 . \nonumber 
 \end{eqnarray}
 We set 
 \[ \beta^2_{-1}=-g_2,\quad \beta^2_0=1-\frac{\Dt}{2}+g_2,  \quad \beta^2_1=0.\]
 Therefore 
 \[ -(1-M_2^\prime(u))(1-\frac{\Dt}{2})- |\Gamma_2^\prime(u)-g_2| = (\Bcal^2_{0})^\prime(u)-\beta^2_{0}<0.\]
 Consequently we have (\ref{mono-21}), (\ref{mono-22}) by (\ref{dx33}), and (\ref{mono-20}) for $l=\pm 1$. It remains to satisfy 
 \begin{eqnarray}
(\Bcal ^2_{0})^\prime(u) \geq 0. \label{mono-220}
\end{eqnarray}
{\bf First case:} $\Gamma^\prime_2(u)-g_2 >0$.

By (\ref{esti-Ap})-(\ref{esti-MP2}) we have
\begin{eqnarray}
(\Bcal ^2_{0})^\prime(u) &\geq &\displaystyle \frac{\lambda-\alpha}{\lambda} \left( 1 -\frac{\Dt}{2}+\frac{\rho \lambda \alpha}{2a}\right)-\frac{\rho \alpha}{a \lambda}(\lambda^2-a^2) \nonumber \\
& \geq & \displaystyle \frac{\lambda-\alpha}{\lambda} \left( 1 - \sigma \rho \right). \nonumber
 \end{eqnarray}
One can prove that  $\sigma > 0$. We obtain (\ref{mono-220}) by (\ref{rho33}). 

 {\bf Second case:} $\Gamma^\prime_2(u)-g_2 <0$.
\begin{eqnarray}
(\Bcal ^2_{0})^\prime(u) &\geq & M^\prime_2(u)\left( 1 -\frac{\Dt}{2}-\frac{\rho \lambda \alpha}{2a}\right) \geq0\nonumber \end{eqnarray}
by (\ref{rho33}).
\end{proof}

\section{The linear case}\label{sec_lin}
\subsection{Numerical schemes}
Here we want to apply the argumentation of the above sections to the linear case, first considered in work \cite{DMR}.
We shall focus on problem \eqref{cv_sys} for $F(u)=a u$. We set $\alpha=\mu a$ and $\beta=\mu$ and we
study the following problem
\begin{equation}\label{probtest}
\left\{\begin{array}{l}
u_t+ \alpha u_x+ z_x=0,\\
\\
z_t+ u_x-\alpha z_x=-\beta z.
\end{array}\right.
\end{equation}
 
The numerical approximation given in \eqref{numcd}, here becomes
\begin{equation}\label{schemeU}
\begin{array}{l}
\Frac{U^{n+1}_j-U^n_j}{\Dt}+\Frac{A}{2\Dx}\left(U_{j+1}^n-U_{j-1}^n\right)
-\Frac{Q}{2\Dx}\left(U_{j+1}^n-2U_j^n+U_{j-1}^n\right)\\
\\ \quad
=\Bcal_{-1} U_{j-1}^n+\Bcal_0 U_j^n+\Bcal_1 U_{j+1}^n,
\end{array}
\end{equation} 
where $U=(u,z)^T$, for $\pm\xi=\pm\sqrt{1+\alpha^2}$ be the eigenvalues of matrix $A$, $Q=diag{(\xi, \xi)}$ and 
$\Bcal_{-1,0,1}=(\beta_{ij}^{-1,0,1})_{i,j=1,2}$ are the matrix of constant coefficients for the source term approximation.

First of all we shall analyze the decay properties of the local truncation error for the numerical approximations described in \cite{DMR}. 

By Taylor expansion, the local truncation error for the numerical approximation (\ref{schemeU}) is 
given by
\begin{equation}\label{trunc_lin}
\left\{\begin{array}{rl}
\Tcal_u=&\frac{\Dt}{2}u_{tt} -\Dx\left[\frac{\xi}{2}u_{xx}
+(\beta_{11}^1-\beta_{11}^{-1})u_x
+(\beta_{12}^1-\beta_{12}^{-1})z_x+c_{11}u+c_{12}z\right]\\ \\ &+\mathcal{O}(\Dx^2+\Dt^2),\\
\\
\Tcal_z=&\frac{\Dt}{2}z_{tt}-\Dx\left[\frac{\xi}{2}z_{xx}
+(\beta_{21}^1-\beta_{21}^{-1})u_x
+(\beta_{22}^1-\beta_{22}^{-1})z_x+c_{21}u+c_{22}z\right]\\ \\ &+\mathcal{O}(\Dx^2+\Dt^2),
\end{array}\right.
\end{equation}
where the $(c_{ij})_{i,j=1,2}$ constants were defined in \cite{DMR}.
Let $\Dt/\Dx=\rho$ be fixed. By relations
$$ u_{tt} = \alpha^2 u_{xx}-(z_t-\alpha z_x)_x, \quad z_{tt} = \xi^2 z_{xx} -\beta (z_t+\alpha z_x),
$$
we get, for some popular schemes, the following expansions.
\\
\begin{itemize}
\item[(UP)] for the source term point wise approximation, we have
\begin{equation}\label{up_lin}
\Bcal_1-\Bcal_{-1}=0, \quad C=0,
\end{equation}
\begin{equation}\label{trunc_up}
\left\{\begin{array}{rl}
\Tcal_u=&\frac{\Dx}{2}\left[ (\rho\alpha^2-\xi) u_{xx} - \rho  (z_t-\alpha z_x)_x \right]+\mathcal{O}(\Dx^2)\ \sim \ \Dx\ t^{-3/2},\\
\\
\Tcal_z=&\frac{\Dx}{2}\left[ \xi(\rho\xi -1) z_{xx} - \rho\beta  (z_t+\alpha z_x) \right]+\mathcal{O}(\Dx^2)\ \sim \ \Dx\ t^{-3/2}.\\
\end{array}\right.
\end{equation}
\item[(ROE)] for the upwinding of the source term,  we have
\begin{equation}\label{g_roe_lin}
 \Bcal_1-\Bcal_{-1}=\Frac{\beta}{2\xi}\left(\begin{array}{cc}
	0 & 1\\
	0  & -\alpha
\end{array}\right), \quad C=0,
\end{equation}
\begin{equation}
\left\{\begin{array}{rl}
\Tcal_u=&\frac{\Dx}{2}\left[(\rho\alpha^2-\frac{(\xi^2-1)}{\xi})u_{xx}+(\frac{1}{\xi}-\rho)\left(z_{t}-\alpha z_{x}\right)_x\right]+\mathcal{O}(\Dx^2)\ \sim \ \Dx\ t^{-3/2},\\
\\
\Tcal_z=&\frac{\Dx}{2}\left[\xi(\rho\xi-1)z_{xx}-\beta(\rho z_t+\alpha(\rho-\frac{1}{\xi})z_x)\right]+\mathcal{O}(\Dx^2)\ \sim \ \Dx\ t^{-3/2}.
\end{array}\right.
\end{equation}
\item[(AHO2p)] for the Asymptotic High Order scheme given in \cite{DMR},  we have
\begin{equation}\label{AHO2p}
 \Bcal_1-\Bcal_{-1}=\Frac{\beta \xi}{2}\left(\begin{array}{cc}
	0 & 1\\
	1  & -2\alpha
\end{array}\right), \quad C=\Frac{\beta^2 \xi}{2}\left(\begin{array}{cc}
	0 & 0\\
	0  & 1
\end{array}\right),
\end{equation}
\begin{equation}
\left\{\begin{array}{rl}
\Tcal_u=&\frac{\Dx}{2}\left[\rho\alpha^2 u_{xx} +(\xi-\rho)\left(z_{t}-\alpha z_{x}\right)_x\right]+\mathcal{O}(\Dx^2)\ \sim \ \Dx\ t^{-3/2},\\
\\
\Tcal_z=&\frac{\Dx}{2}\left[\xi(\rho\xi-1)z_{xx}-\beta(\rho+\xi)(z_t+\alpha z_x)\right]+\mathcal{O}(\Dx^2)\ \sim \ \Dx\ t^{-3/2}.
\end{array}\right.
\end{equation}
\end{itemize}

Let us now go back to the Time-AHO schemes.
According to the discussion of section \ref{sec_decayTronc_gen}, conditions given in propositions \ref{prop_aho_gen}-\ref{prop_monot}, here give rise to the following choice,  
\begin{equation}\label{td_lin}
 \Bcal_1-\Bcal_{-1}=\Frac{\beta\rho}{2}\left(\begin{array}{cc}
	0 & \frac{\xi}{\rho}-\alpha^2\\
	1  & -2\alpha
\end{array}\right),\quad\quad C=\Frac{\beta^2 \rho}{2}\left(\begin{array}{cc}
	0 & 0\\
	0  & 1
\end{array}\right).
\end{equation}
For the local truncation error it then holds,
\begin{equation}\label{decayfinal}
\left\{\begin{array}{rl}
\Tcal_u=&-\Frac{\Dx}{2}\left[\xi(1-\rho\xi)\left(-z_{xt}+\alpha z_{xx}\right)\right]+\mathcal{O}(\Dx^2)\ \sim \ \Dx\ t^{-2}
,\\
\\
\Tcal_z=&-\Frac{\Dx}{2}\left[\xi(1-\rho\xi) z_{xx}\right]+\mathcal{O}(\Dx^2)\ \sim \ \Dx\ t^{-2}.
\end{array}\right.
\end{equation}

\begin{remark}\label{remark-lin-a0}
 As for the non-linear case, for $\alpha=0$, the Time-AHO scheme
 reduces to the Roe approximation, which in that case decays like in \eqref{decayfinal}.
\end{remark}

\subsection{Modified equation}\label{ahoParab}
Here we wish to better understand the qualitative behavior of the numerical methods described in the previous section by applying the {\it modified equation method} (see for instance \cite{LeVbook92}). From Taylor expansion stated in \eqref{trunc_lin}, for all numerical schemes described in Section \ref{sec_lin}, we get the following modified system
\begin{equation}\label{modif_eq}
\left\{\begin{array}{rl}
u_t+\alpha u_x+z_x=&-\frac{\Dt}{2}u_{tt} +\Dx\left[\frac{\xi}{2}u_{xx}
+\gamma_{11} u_x+\gamma_{12}z_x+c_{11}u+c_{12}z\right],\\
z_t+u_x-\alpha z_x+\beta z=&-\frac{\Dt}{2}z_{tt}+\Dx\left[\frac{\xi}{2}z_{xx}
+\gamma_{21} u_x
+\gamma_{22} z_x+c_{21}u+c_{22}z\right].
\end{array}\right.
\end{equation}
Since we are  interested in long time simulations, we apply the decay rates results given in the analytical Section \ref{decays} to the derivatives terms of problem \eqref{modif_eq}. As time increases, we can then take into account as modified equation the following {\it asymptotic modified system},   
\begin{equation}
\left\{\begin{array}{rl}
u_t+\alpha u_x+z_x=&-\frac{\Dt}{2}u_{tt} +\Dx\left[\frac{\xi}{2}u_{xx}
+\gamma_{11} u_x+\gamma_{12}z_x+c_{11}u+c_{12}z\right],\\
u_x+\beta z=&\Dx\left[\gamma_{21} u_x +c_{21}u+c_{22}z\right].
\end{array}\right.
\end{equation}
For all schemes described in Section \ref{sec_lin}, we get in the second equation $\left[\gamma_{21} u_x +c_{21}u+c_{22}z\right]=0$ and then, for all of them the asymptotic modified problem becomes
\begin{equation}\label{modif_eq_asympt}
\left\{\begin{array}{rl}
u_t+\alpha u_x=&\left(\frac{1}{\beta}+\Dx \Dcal\right) u_{xx} + O(\Dx^2),\\
z =& -\frac{1}{\beta}u_x + O(\Dx^2),
\end{array}\right.
\end{equation}
where the constant $\Dcal$ depends on the selected scheme.
\\
Specifically, for UP, ROE and AHO2p we obtain a perturbation of order $O(\Dx)$ on the diffusion term with
$$ \Dcal_{UP} = -\Frac{1}{2}\left(\rho\alpha^2-\xi\right), \quad \Dcal_{ROE} = -\Frac{1}{2}\left(\rho\alpha^2-\Frac{\xi^2-1}{\xi}\right), \quad \Dcal_{AHO2p} = -\Frac{\rho\alpha^2}{2};$$
while, for TAHO we get 
$$\Dcal_{TAHO}=0.$$

The TAHO approximation is then second order accurate on the asymptotic diffusion (Chapman-Enskog) limit
\begin{equation}\label{parabtest}
\left\{\begin{array}{l}
\hat u_t+ \alpha \hat u_x=\frac{1}{\beta}\hat u_{xx},\\
\hat z=-\frac{1}{\beta}\hat  u_x.
\end{array}\right.
\end{equation}

\section{Numerical tests}\label{Sec5}
In this section we show how, for large time simulations,
Time-AHO schemes give better numerical results than standard approximations both for linear and non-linear cases.
\\
From now on we shall call STD the following standard pointwise upwind approximation, for $i=1,...,m$
\begin{equation}\label{std_gen}
\begin{array}{l}
\Frac{f^i_{j,n+1}-f^i_{j,n}}{\Dt}+\Frac{\lambda_i}{2\Dx}\left(f^i_{j+1,n}-f^i_{j-1,n}\right)-\Frac{|\lambda_i|}{2\Dx}\delta_x^2 f^i_{j,n}=M_{i}(u_{j,n})-f^i_{j,n}.
\end{array}
\end{equation} 

For all tests, we 
focus our attention on the numerical error as a function of time:
denoting $(u,Z)$ the conservative-dissipative variables, 
we plot the errors $e_u(t)=\| (u^H-U^h)(t)\|_{L^\infty}$, $e_z(t)=\| (Z^H-Z^h)(t)\|_{L^\infty}$ as the time  $t=n\Dt$ increases, 
where  $(u^H,Z^H)$ is the reference solution obtained by the ROE scheme with  $\Dx=\mathcal{O}(10^{-4})$.

For all schemes, we fix the steps ratio $\rho$ to verify all the CFL conditions;
Since all schemes are of first order approximation, to emphasize the good behaviour of TAHO compared to the others schemes, we compute the numerical solutions $U^h$ by using a quite big grid step $\Dx=\mathcal{O}(10^{-1})$.

We then plot the different approximations of functions $u$ and $Z$ at final time $T=450$, focusing on the point of maximum value of the solution to highlight the differences of the approximations. 
Near it, we show the most interesting plot of the $l^\infty$ errors as a function of time.

Then, given different numerical approximations $U^h$, we look for
constant $C_u$, $\gamma_u$,  $C_z$, $\gamma_z$ which best fit the equality
\begin{equation}\label{decaytest}
e_u(t)=\| (u^H-U^h)(t)\|_{L^\infty}= C_u t^{-\gamma_u}, \quad e_z(t)=\| (Z^H-Z^h)(t)\|_{L^\infty}= C_zt^{-\gamma_z}.
\end{equation}
Given $N$ data points $(t_i,e(t_i))_{i=1,N}$, we shall define $\gamma$ and $C$ as the solution of the following least squares problem,
$$\min_{C,\gamma}\Sum_{i=1}^N |\ln(e(t_i))-\ln(Ct^{-\gamma}) |^2.$$

All numerical results we present show that for standard approximations, such as upwind \eqref{std_gen} and ROE \eqref{roe_gen}, the absolute error $e(t)$, for a fixed space step, decays as
\[ e_u(t)= O( t^{-1/2}), \ e_z(t)= O( t^{-1}),\]
while for the TAHO scheme, it improves of $t^{-1/2}$ on the previous schemes.

\subsection{Results for the linear case}
Let us consider  for system (\ref{probtest}),  the  constant equilibrium state $u=1$ and $z=0$.
As in our previous work \cite{DMR}, we fix $\beta=5$ and we consider a small compactly supported perturbation of this constant solution as initial data.
\begin{equation}
u_0 = \chi_{[-1,1]} \left(- x^2 + 2\right), \quad z_0 = \chi_{[-1,1]} \left(- x^2 + 1\right).
\end{equation}
We then compare the Time-AHO scheme \eqref{td_lin} with the following schemes: the AHO2p scheme \eqref{AHO2p}, the standard first-order point wise upwind scheme \eqref{up_lin} and the ROE scheme \eqref{g_roe_lin}.

\subsubsection{Test case with $\alpha = 1$}\label{linT2}
As expected by our asymptotic analysis, the numerical results show a better performance of the TAHO scheme. In Figure \ref{fig-lin1.1}-(a)-(b) we plot a zoom on the solutions $u$ and $z$ respectively, obtained by applying the different numerical schemes at final time $T$. The solution given by TAHO follows much better than the other the benchmark curve.
 
Moreover, always in Figure \ref{fig-lin1.1}, the two plots $(c)$ and $(d)$ show the time evolution of the $l^\infty$ errors $e_u(t)$ and $e_z(t)$ defined in \eqref{decaytest} for all schemes considered; They show how for the TAHO scheme, as time increases, both errors decay more quickly than the other. 
This result is also confirmed by Table \ref{tab-lin_1.1}, where the values of $\gamma$ and $C$ are computed. Looking at the different values of $\gamma$, it is clear that for the TAHO approximation the decay velocity of the absolute error improves of $t^{-1/2}$ on the previous schemes.

To stress on the good behavior of the TAHO scheme, in Figure \ref{fig-lin2}, we plot the solution $u$ obtained by different numerical approximations with decreasing space step $\Dx$. All schemes considered are of first order approximation, but looking at the numerical curves, it is clear how for large step the TAHO solution follows better the benchmark line than the other.

\begin{figure}[ht]
\scalebox{0.5}{
\begin{tabular}{cc}
\includegraphics{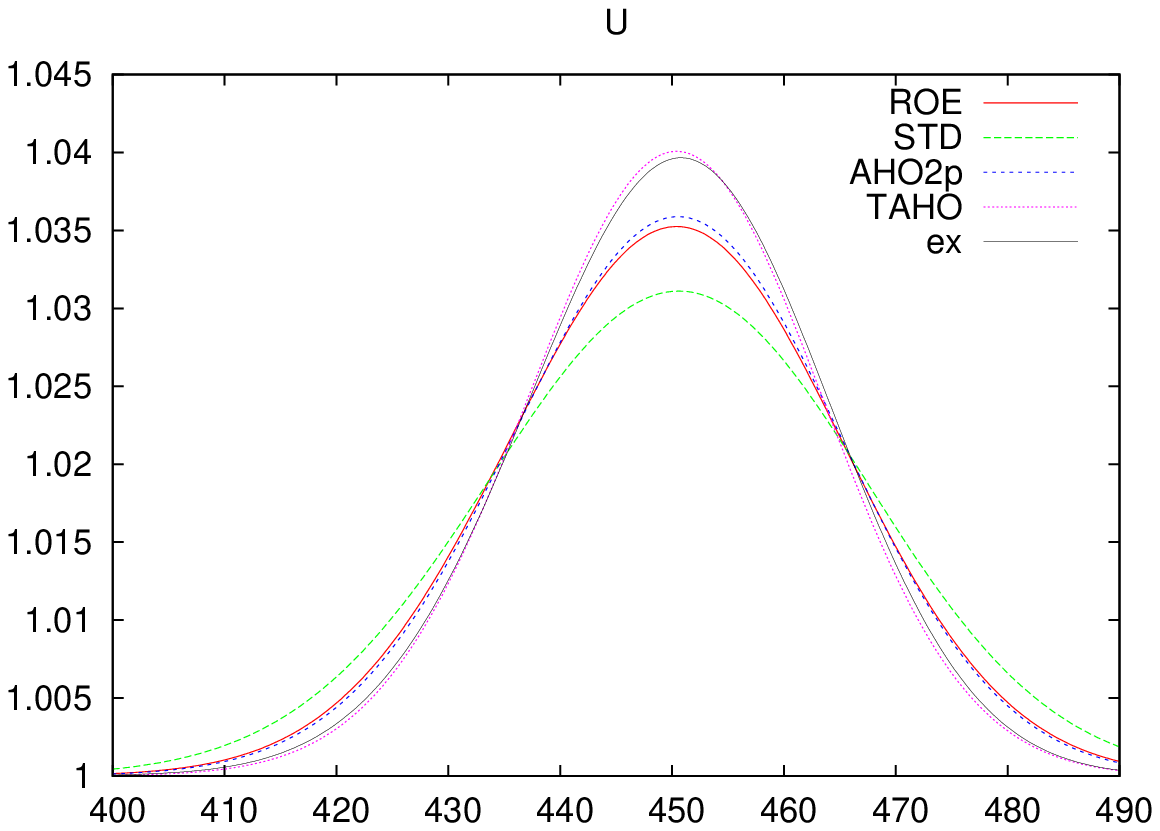}
& 
\includegraphics{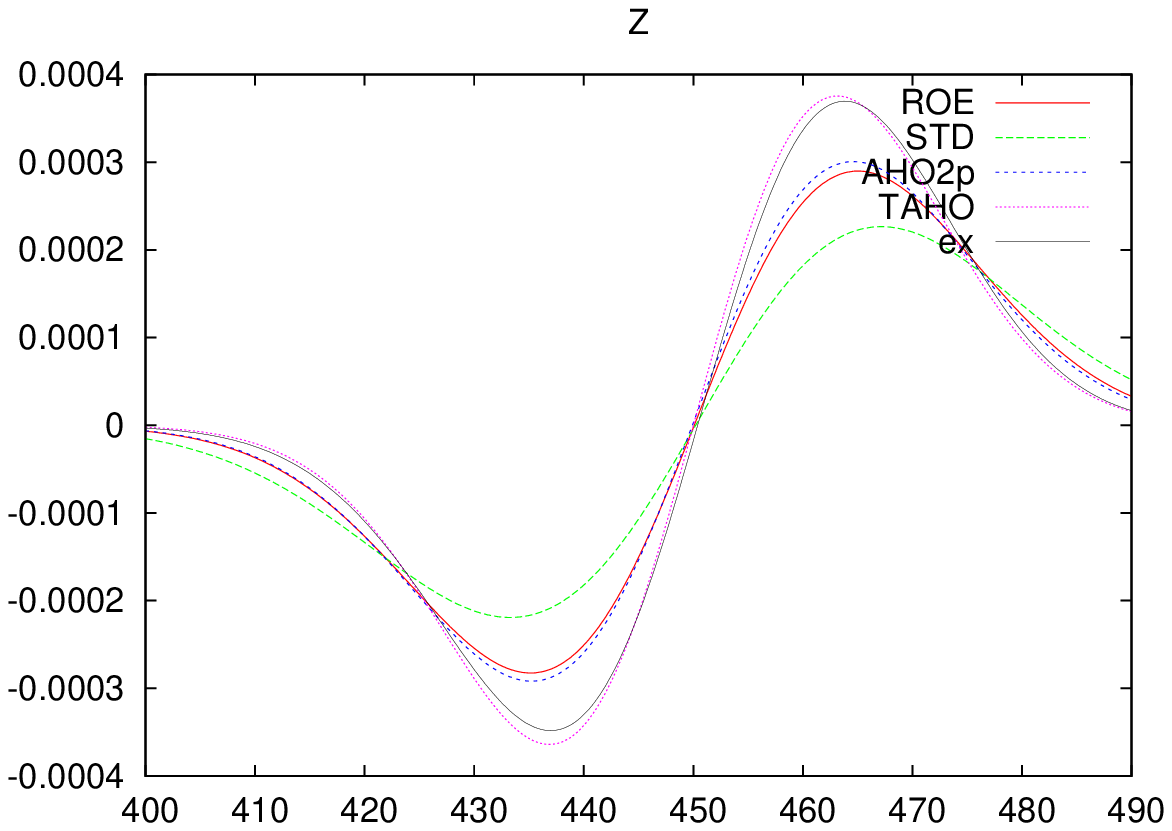}\\
(a) & (b) \\
\includegraphics{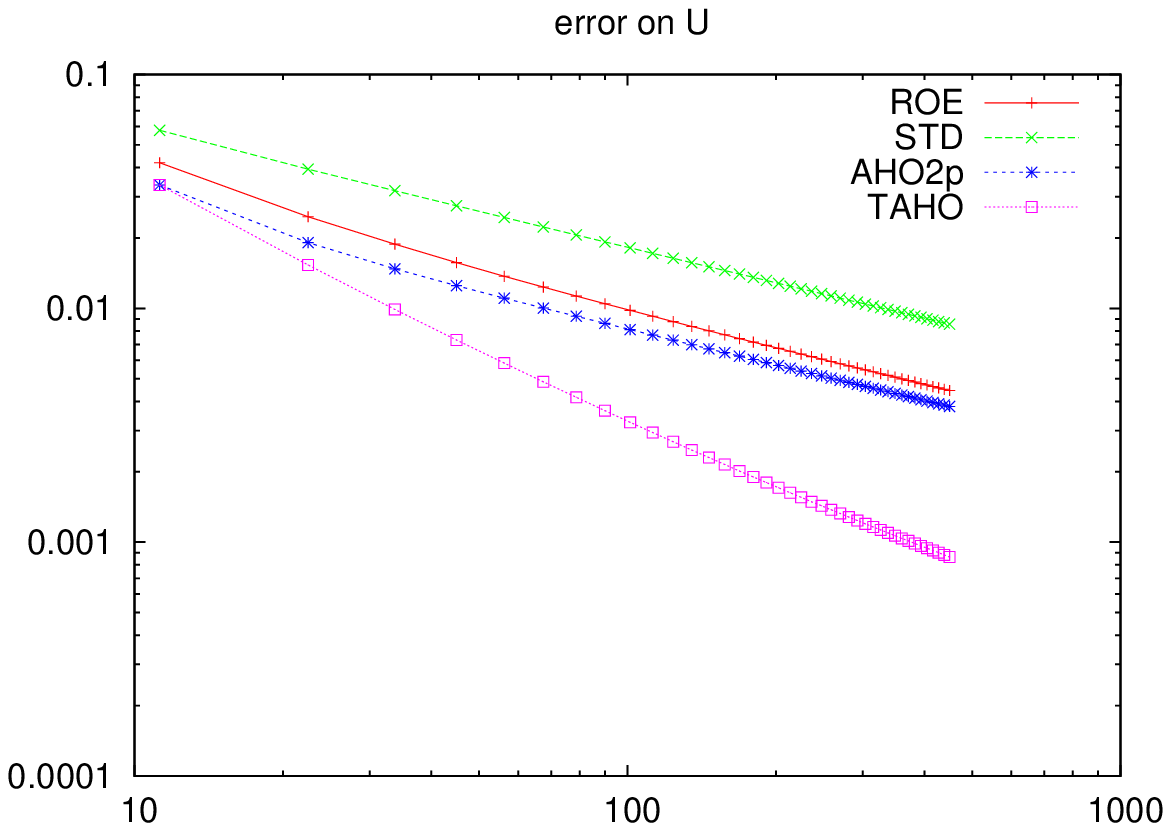}
&
\includegraphics{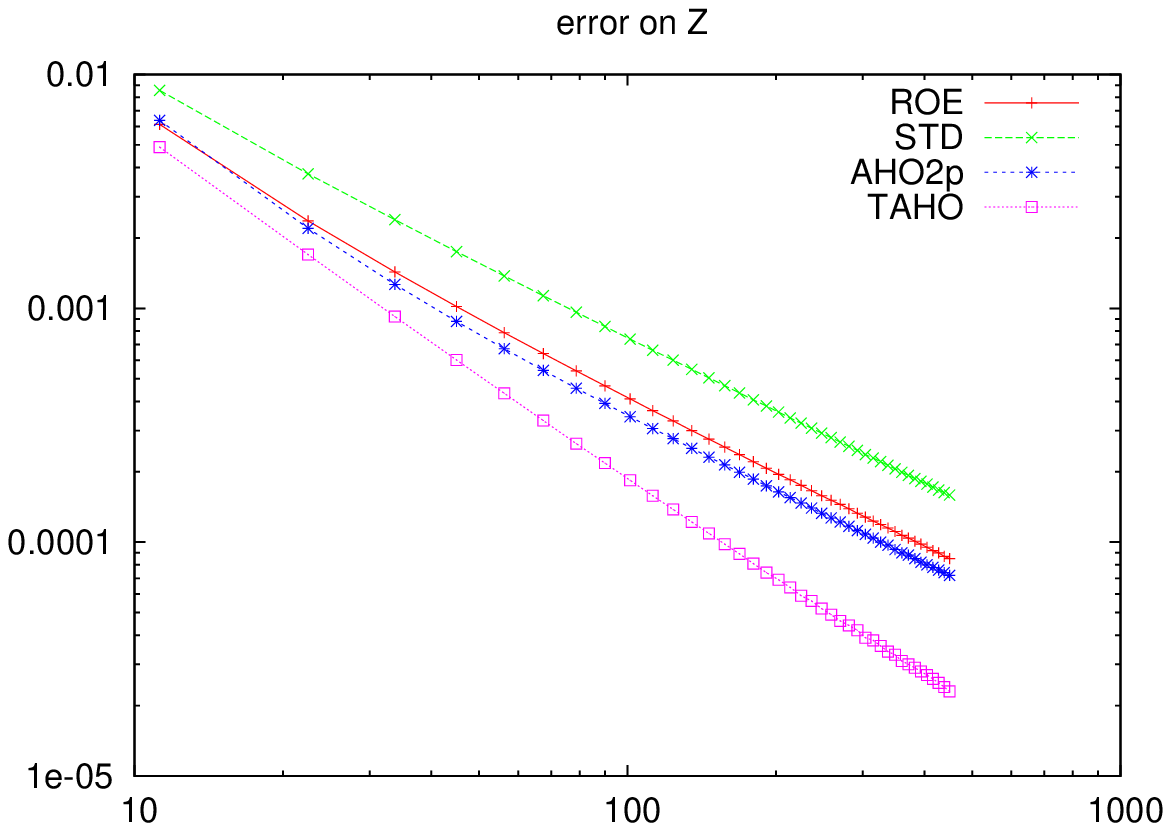}\\
(c) & (d)
\end{tabular}}
\caption{Linear Case Test, section \ref{linT2}. (a)-(b) zoom on the solutions $u$ and $z$ respectively obtained by the different schemes at final time $T$. The plot show a better performance of TAHO scheme. (c)-(d) time evolution of the $l^\infty$ errors $e_u(t)$ and $e_z(t)$ defined in \eqref{decaytest} for the different schemes. As expected by our asymptotic analysis, for the TAHO scheme the absolute errors $e_{u,z}(t)$ decay faster as the time increases.
}
\label{fig-lin1.1}
\end{figure}

\begin{table}[ht]
\begin{center}
\begin{tabular}{|c|c|c|c|c|}\hline
\mbox{scheme} & $C_u$ &  $\gamma_u$ & $C_z$ & $\gamma_z$ \\ \hline
STD & 0.192993 & 0.510798 & 0.100412 &  1.058824 \\\hline
ROE & 0.143847 & 0.573922 & 0.073287 &  1.112843 \\\hline
AHO2p & 0.104897 & 0.547166 & 0.072874 &  1.143070 \\\hline
TAHO & 0.289768 & 0.961310 & 0.141281 &  1.432194 \\\hline
\end{tabular}
\caption{Linear Case Test, section \ref{linT2}. Evaluation of constants $\gamma$ and $C$ for $e_u(t)= C_u t^{-\gamma_u}$ and $e_z(t)= C_z t^{-\gamma_z}$ defined in \eqref{decaytest}. For standard approximations, such as STD and ROE , the numerical results show that the absolute error decays as $e_u(t)= O( t^{-1/2})$ and $e_z(t)= O( t^{-1})$; while, for the TAHO scheme, it improves of $t^{-1/2}$ on the previous schemes.
}
\label{tab-lin_1.1}
\end{center}
\end{table}

\begin{figure}[ht]
\scalebox{0.5}{
\begin{tabular}{ll}
\includegraphics{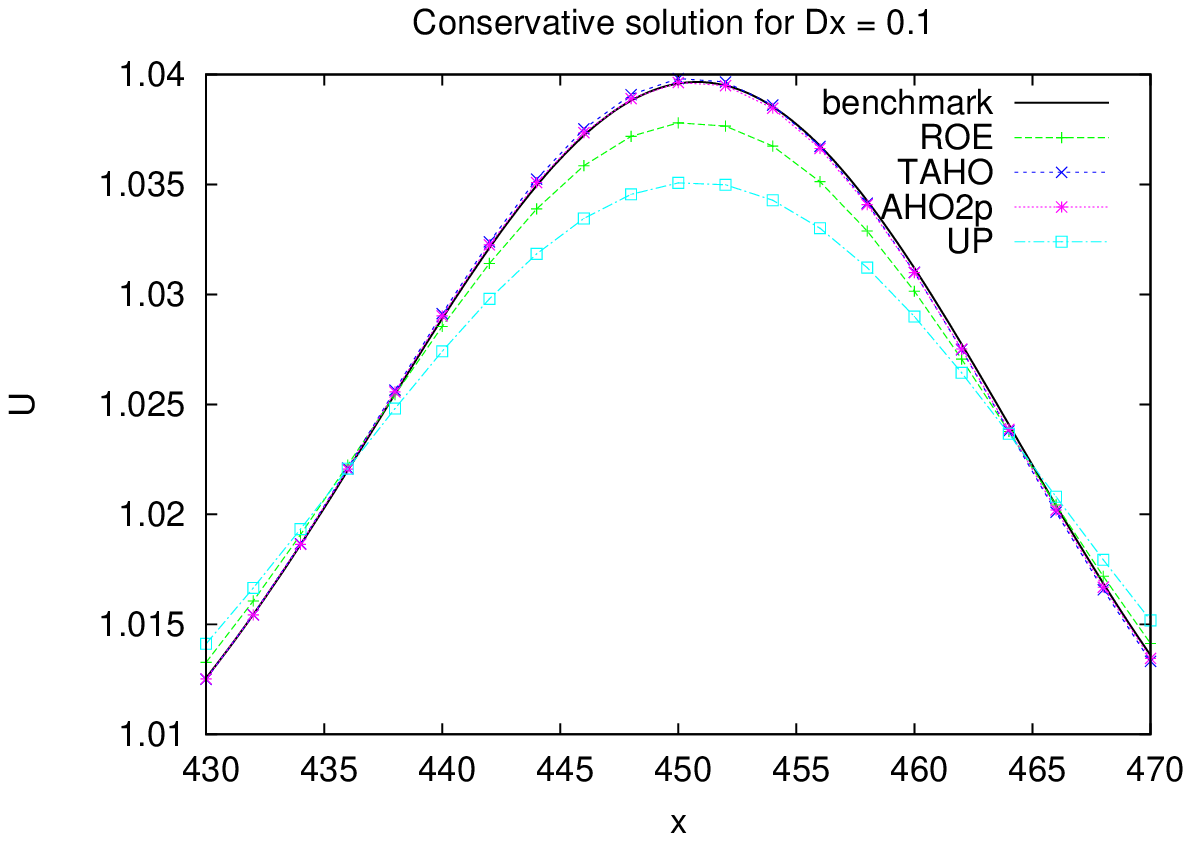}
&
\includegraphics{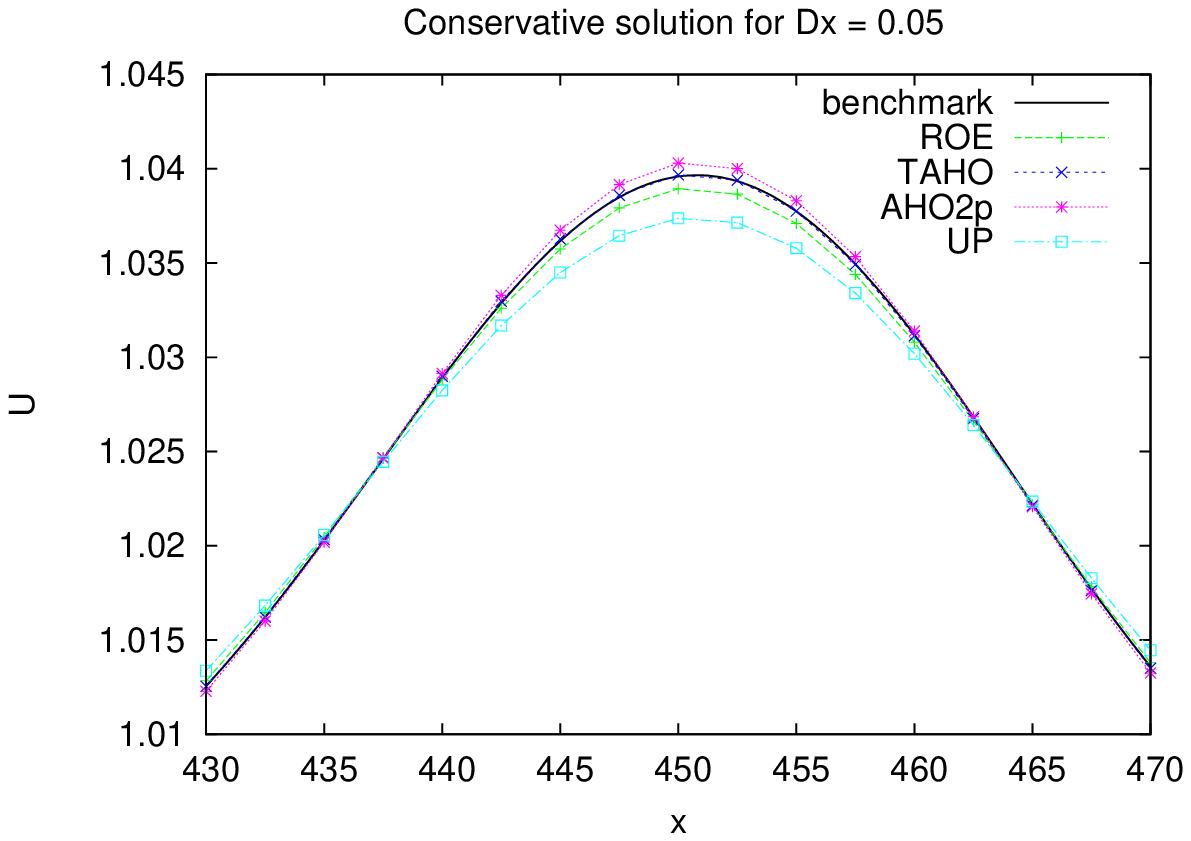}
\\
(a) & (b)\\
\includegraphics{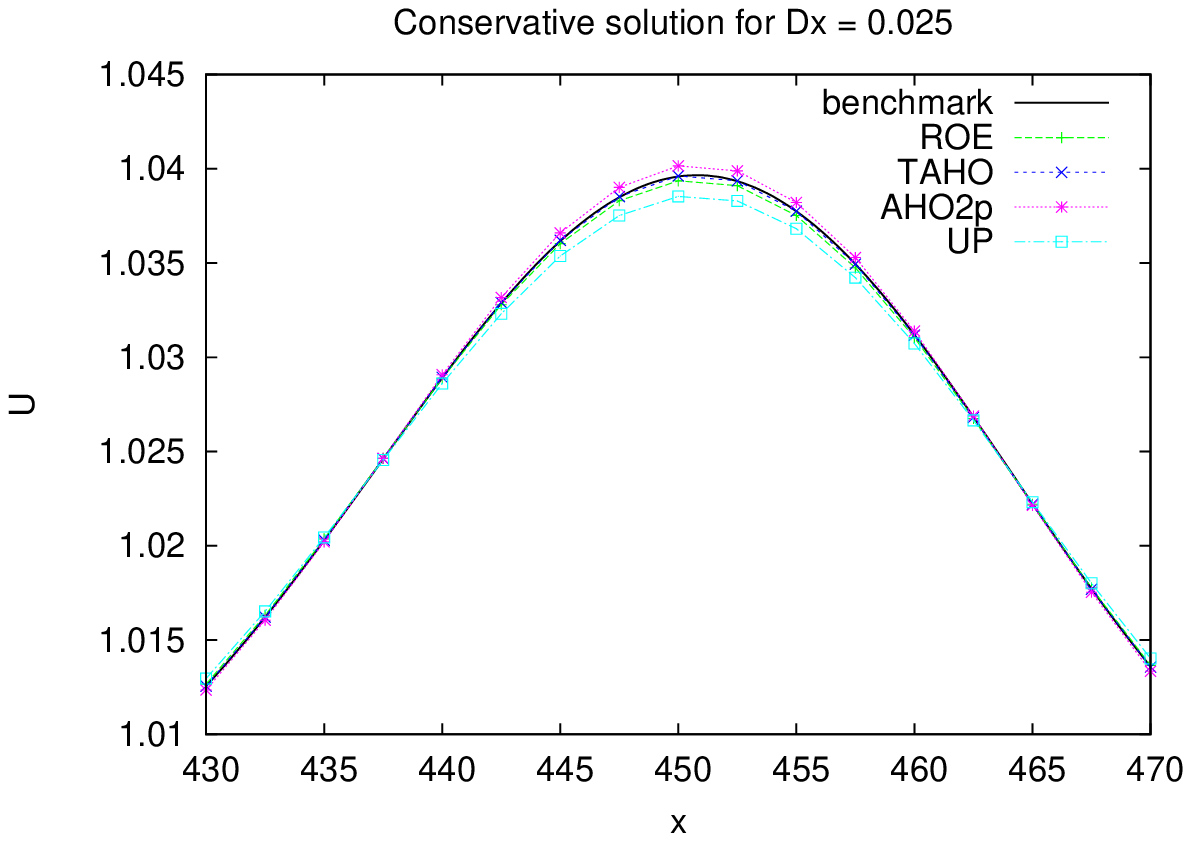}
&
\includegraphics{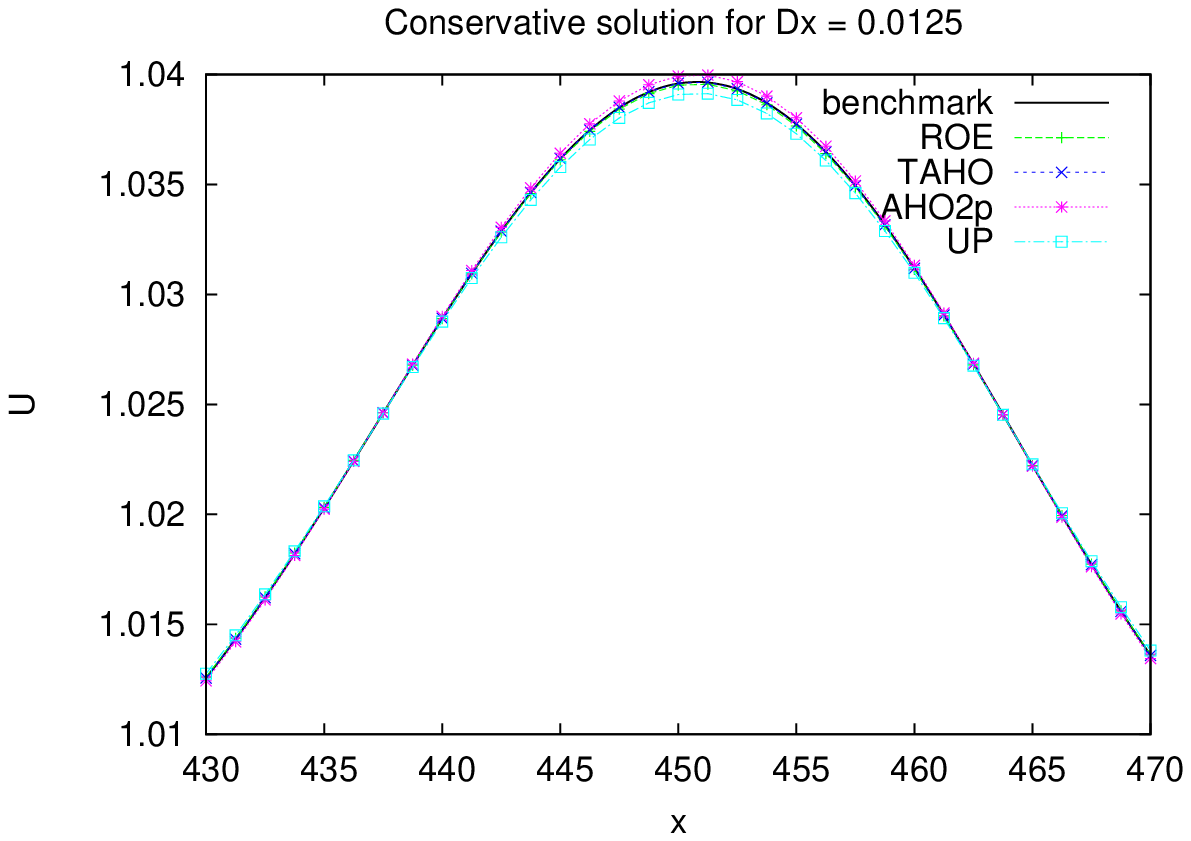}
\\
(c) & (d)
\end{tabular}}
\caption{Linear Case Test, section \ref{linT2}. Zoom on the solutions $U$ obtained by the different schemes at final time $T$ by applying decreasing values of $\Dx$: (a) $\Dx=0.1$, (b) $\Dx = 0.05$, (c) $\Dx = 0.025$ and (d) $\Dx=0.0125$. The numerical solution obtained by TAHO scheme follows better than the others the benchmark curve already with a quite big space step. 
}
\label{fig-lin2}
\end{figure}

\subsubsection{Test case with $\alpha = 0$}\label{linT1}
 As previously observed in Remark \ref{remark-lin-a0}, for $\alpha=0$ the ROE scheme \eqref{g_roe_lin} corresponds to our TAHO approximation. 

For this particular case we can compare the TAHO/ROE scheme with the well-balanced approximation proposed by Gosse and Toscani in  \cite{GT02}, which however is defined only in the case $\alpha=0$. From now on we shall refer to this scheme as WB-GT.

Figure \ref{fig-lin_a0} shows the performances of TAHO/ROE, STD and WB-GT scheme for problem 
\begin{equation}
\left\{\begin{array}{l}
u_t + z_x =0,\\
z_t + u_x =-\beta z,
\end{array}\right.
\end{equation}
with $\beta=5$, at final time $T=450$.

We observe that both TAHO/ROE and WB-GT give better performances than STD. The two numerical solutions obtained by TAHO/ROE and WB-GT showed respectively in Figure \ref{fig-lin_a0}-(a)-(b) are overlapping and the numerical errors in Figure \ref{fig-lin_a0}-(c)-(d) have a similar trend. 
\\
Indeed, according to the discussion of section \ref{sec_decayTronc_gen}, for the local truncation error of the WB-GT approximation it can be shown that 
$$ \Tcal_u \sim \ \Dx \ t^{-2}, \quad \Tcal_z \sim \ \Dx \ t^{-3/2}.$$ 
Then, we observe that the WB-GT scheme is Time-AHO with respect to the conservative variable $u$, while for the dissipative one it is not.
The good behavior of WB-GT scheme may be confirmed by analyzing its order of convergence with respect the asymptotic modified problem \eqref{modif_eq_asympt}. As done in section \ref{ahoParab},
when $t$ goes to $+\infty$, it is possible to show that the WB-GT scheme is second order accurate with respect the asymptotic modified problem
\begin{equation}
\left\{\begin{array}{l}
u_t =  \frac{1}{\beta}u_{xx} + \Frac{\beta\Dx^2}{4(1+\beta\Dx/2)} u_{xx}, \\
z= -\frac{1}{\beta}u_x,
\end{array}\right.
\end{equation}
namely it is of second order with respect the Chapman-Enskog limit. Notice that, for $\alpha\neq0$, it is possible to consider other Well Balanced scheme as in 
\cite{Go00}, but they are not THAO and their asymptotic performances are not of the same order (not shown). 

\begin{figure}[ht]
\scalebox{0.5}{
\begin{tabular}{cc}
\includegraphics{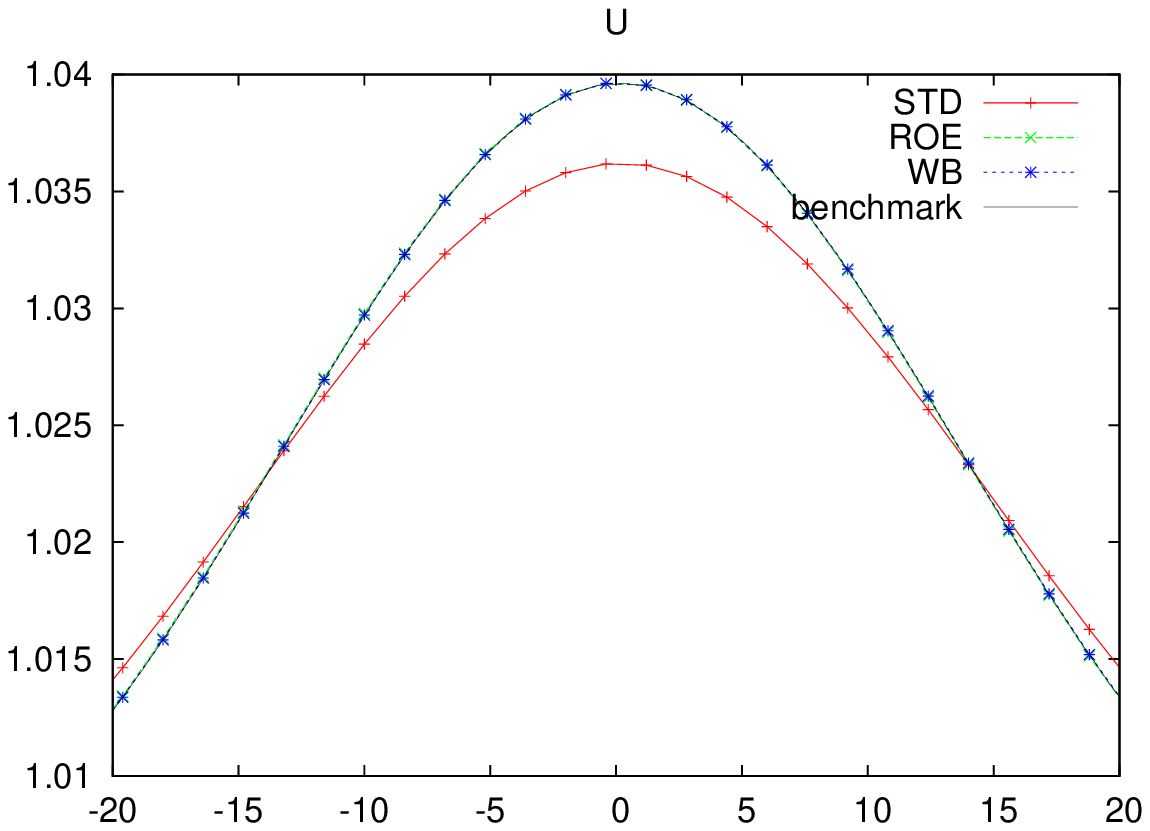}
& 
\includegraphics{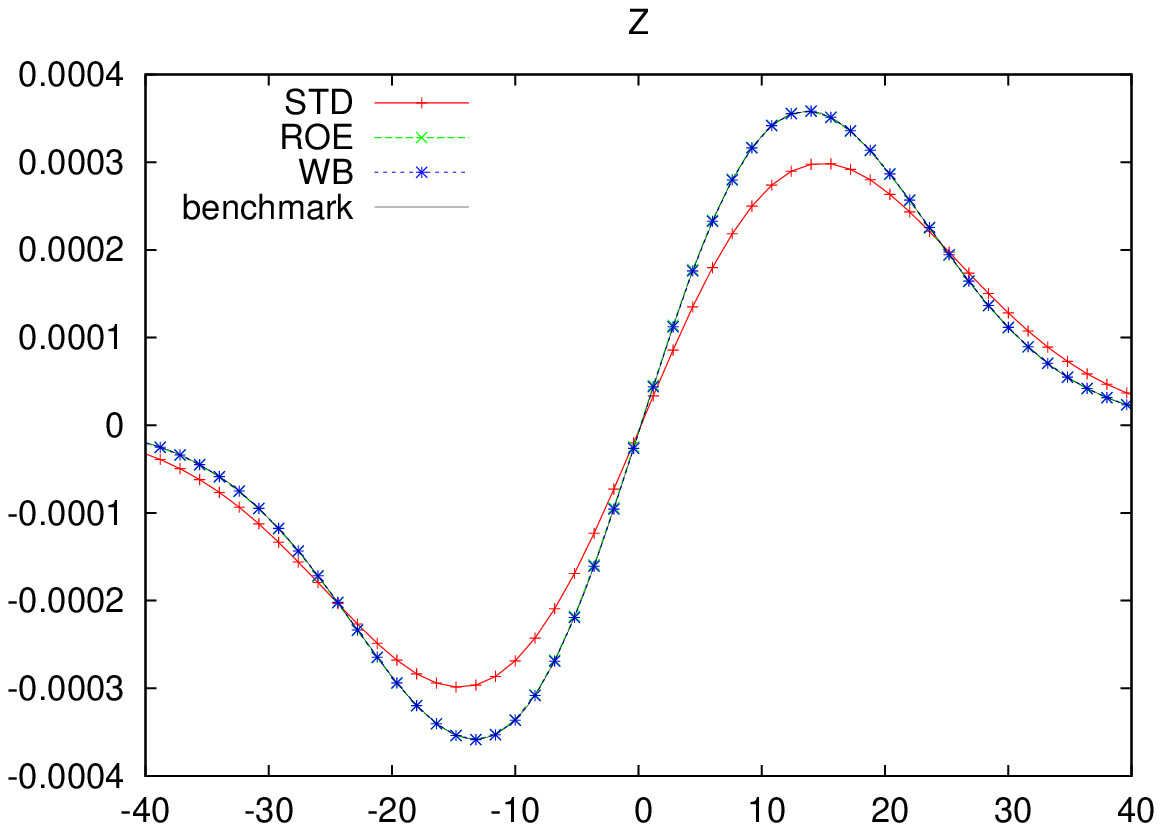}\\
(a) & (b) \\
\includegraphics{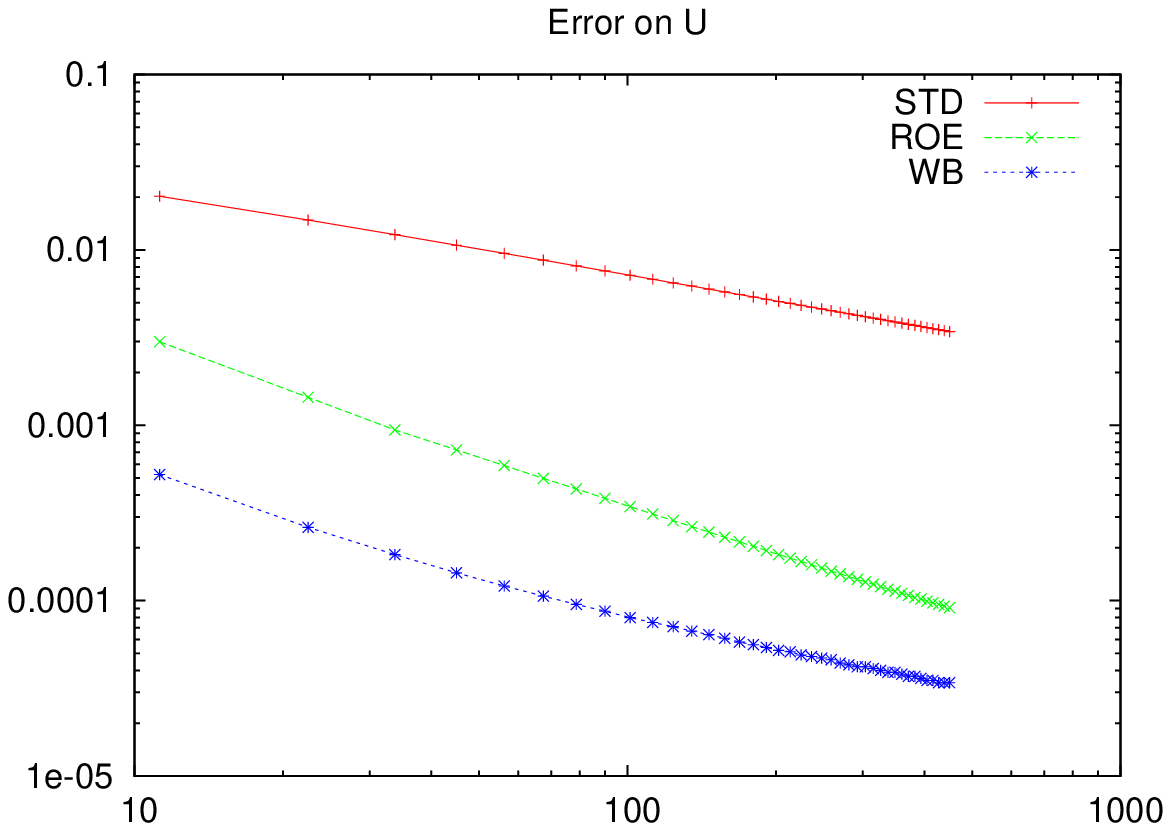}
&
\includegraphics{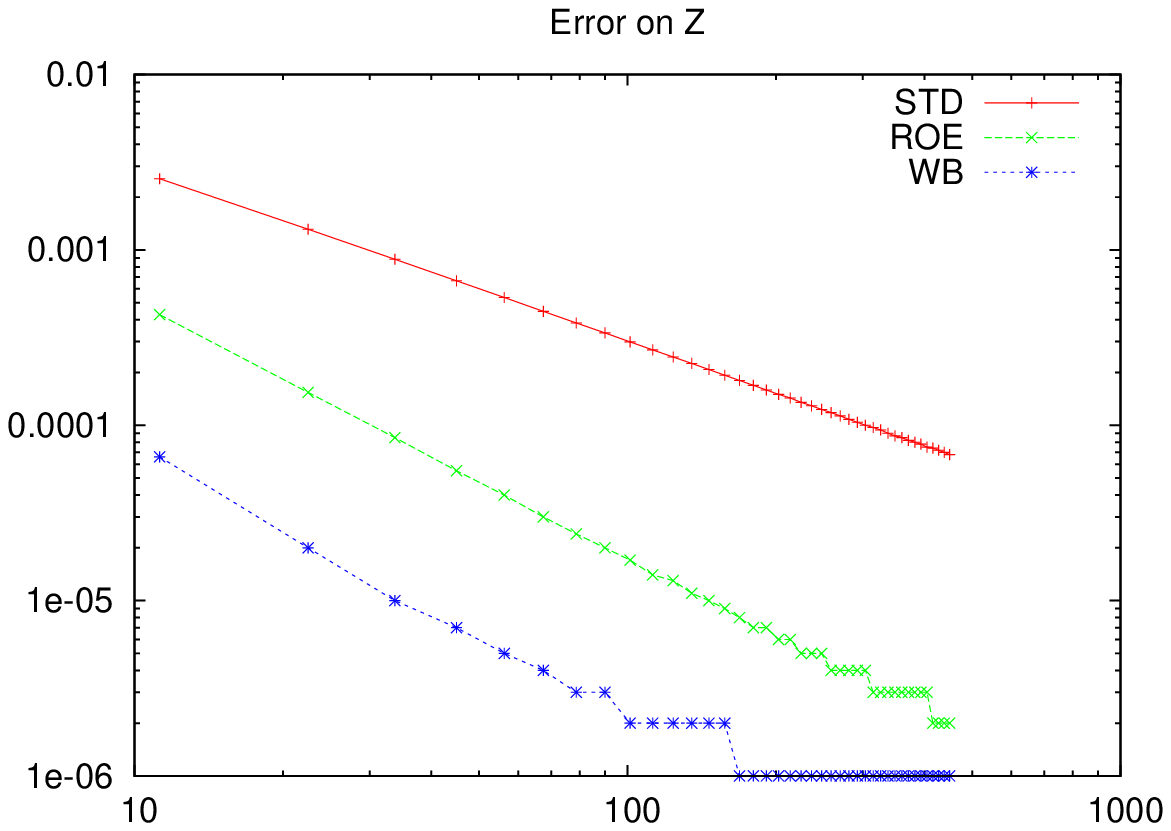}\\
(c) & (d)
\end{tabular}}
\caption{Linear Case Test for $\alpha=0$, section \ref{linT1}. (a)-(b) zoom on the solutions $u$ and $z$ respectively obtained by the different schemes at final time $T$ with $\Dx=0.08$. The plot show a better performance of TAHO/ROE and WB-GT scheme, with respect the STD approximation. As described in section \ref{linT1} WB-GT scheme is Time-AHO with respect the conservative variables and, as TAHO/ROE, it is of order $O(\Dx^2)$ with respect the parabolic asymptotic state. (c)-(d) time evolution of the $l^\infty$ errors $e_u(t)$ and $e_z(t)$ defined in \eqref{decaytest} for the different schemes. As expected by our asymptotic analysis, for the TAHO/ROE and WB-GT schemes, the absolute errors $e_{u,z}(t)$ decay faster as the time increases.
}
\label{fig-lin_a0}
\end{figure}

\subsection{Results for the non-linear $2\times2$ test case}
We fix $q=\lambda$ and we compare for the $2\times2$ case the scheme TAHO \eqref{aho} with ROE \eqref{roe_gen} and STD \eqref{std_gen}.
\\
We shall consider the two different cases $F'(0)\ne 0$ and $F'(0)=0$ and we select as initial datum the function 
\begin{equation}
u_0 = \chi_{[-1,1]} \left(- x^2 + 1\right), \quad Z_0 = \Frac{1}{\lambda} F(u_0(x)).
\end{equation}

\subsubsection{The case $F'(0)\ne 0$}\label{nLinT2}
Here we fix 
$$ F(u)= a \left(u - u^2\right) $$
and we compare our TAHO approximation \eqref{aho} with STD and ROE schemes, defined in \eqref{std_gen} and \eqref{roe_gen} respectively.

The numerical results show a better performance of the TAHO scheme. In Figure \ref{fig-nlin2}-(a)-(b) we plot a zoom on the solutions $u$ and $Z$ respectively, obtained by the different schemes at final time $T$. The solution given by applying the TAHO scheme follows much better than the other the benchmark curve. We stress on that the numerical solutions are computed with quite big step $\Dx=0.1$. 
\\
Moreover, always in Figure \ref{fig-nlin2}, the two plots $(c)$ and $(d)$ show the time evolution of the $l^\infty$ errors $e_u(t)$ and $e_z(t)$ defined in \eqref{decaytest} for all schemes considered; They show how for the TAHO scheme both errors decay as time increases more quickly than others. 
This result is also confirmed by Table \ref{tab-nlin_2}, where the values of $\gamma$ and $C$ are computed. Looking at the different values of $\gamma$, it is clear that for the TAHO approximation the decay velocity of the absolute error improves of $t^{-1/2}$ on the previous schemes.
\begin{figure}[ht]
\scalebox{0.5}{
\begin{tabular}{cc}
\includegraphics{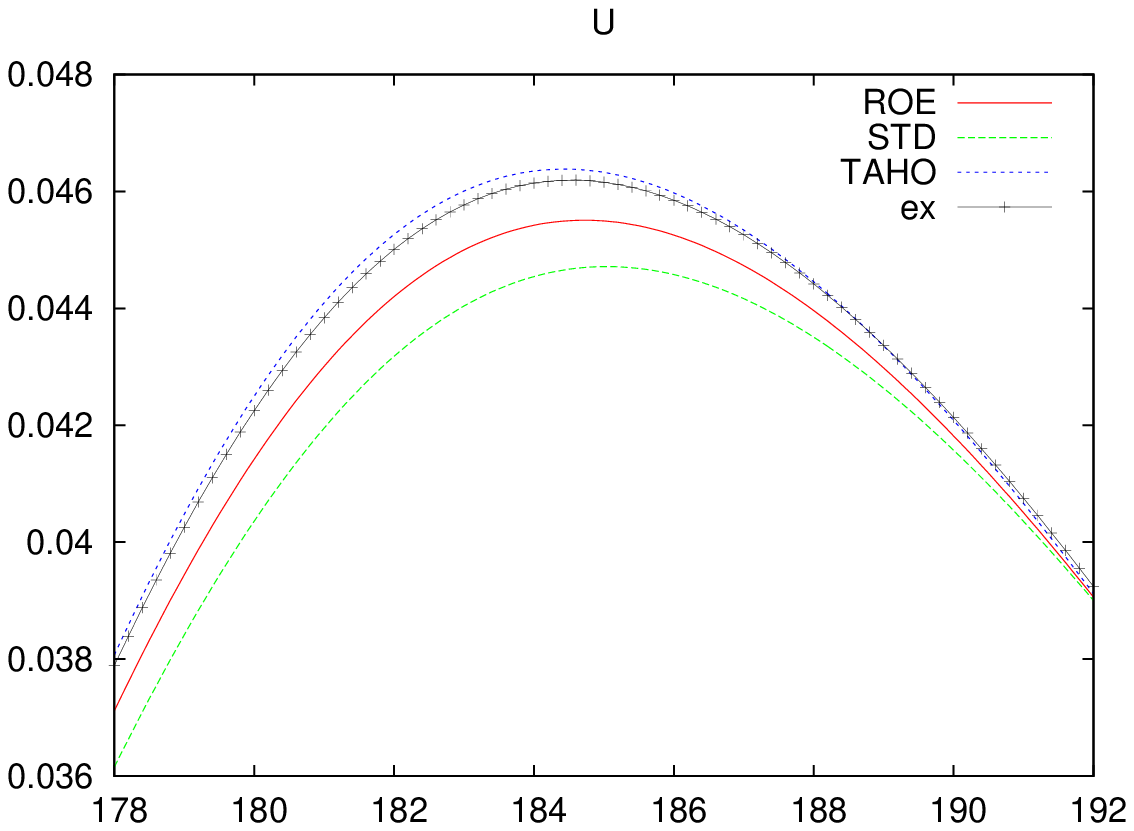}
&
\includegraphics{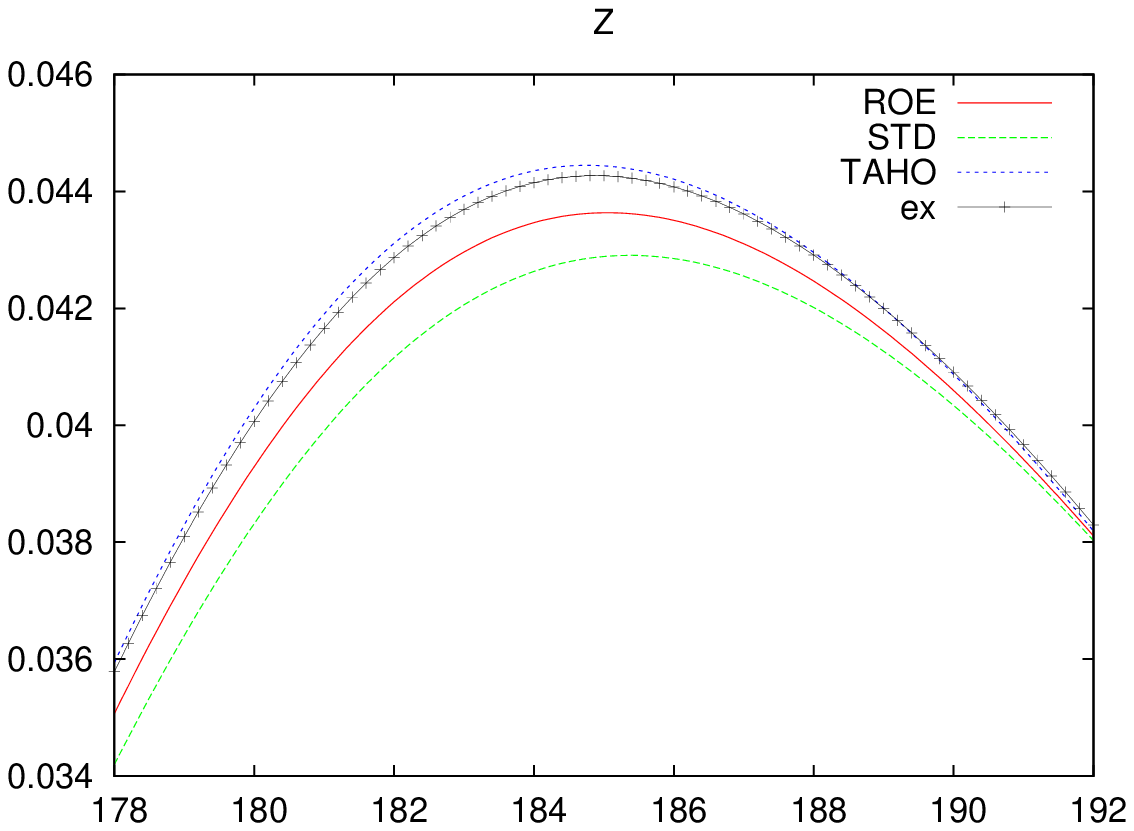}\\
(a) & (b) \\
\includegraphics{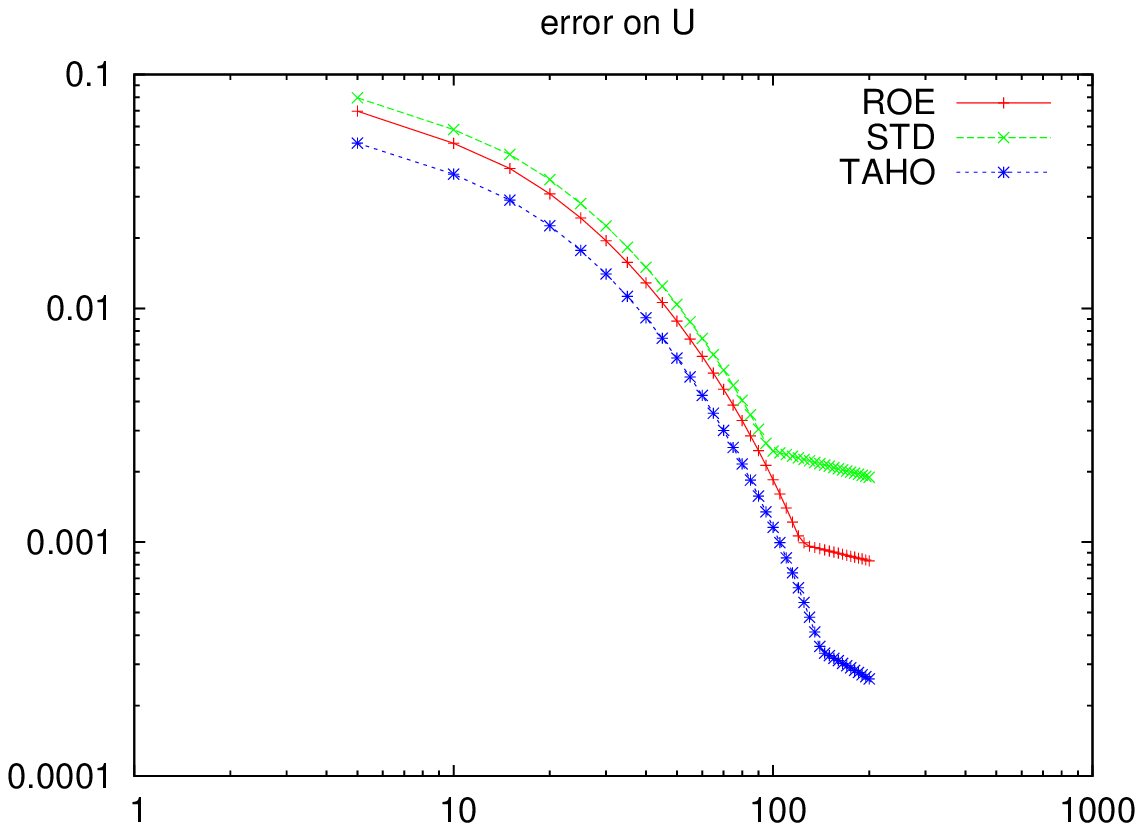}
&
\includegraphics{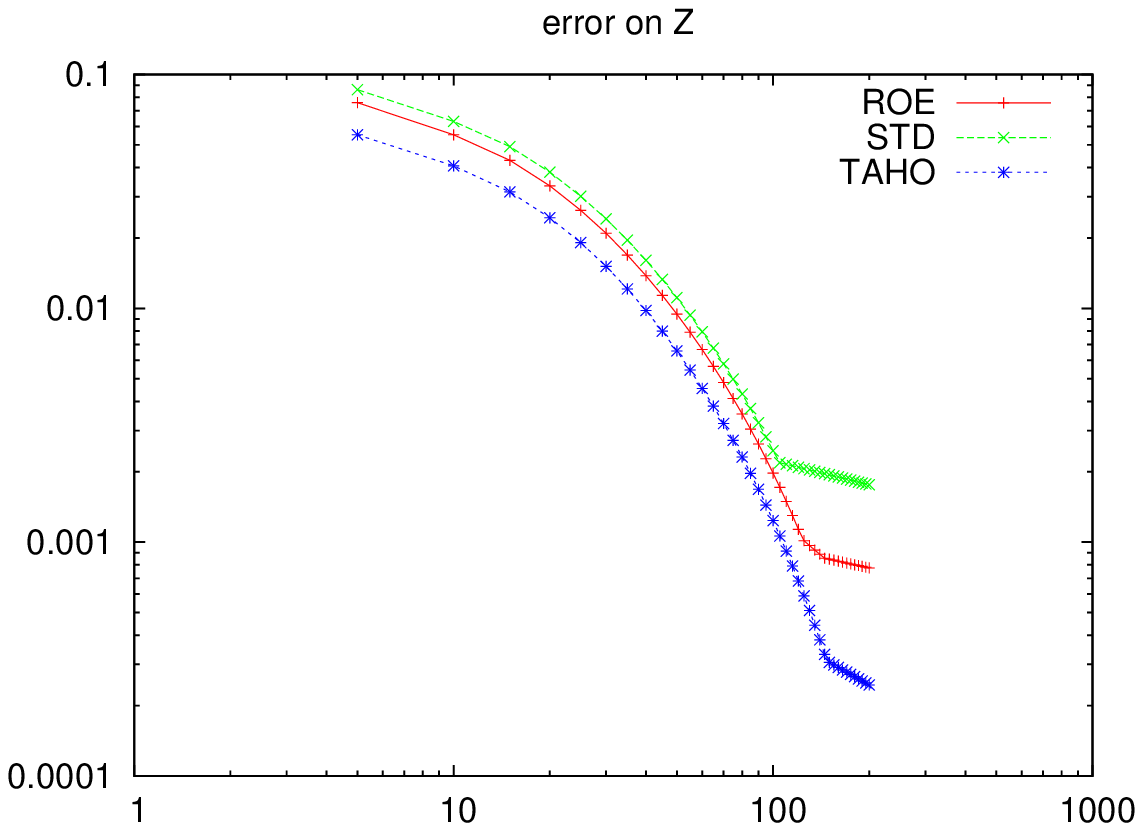}\\
(c) & (d)
\end{tabular}}
\caption{Non-Linear Case Test with $F'(0)\ne 0$, see section \ref{nLinT2}. (a)-(b) Zoom on the solutions $u$ and $Z$ respectively obtained by the different schemes at final time $T$. The plot show that TAHO scheme gives better results than others with a quite big step $\Dx=0.1$. (c)-(d) Time evolution of the $l^\infty$ errors $e_u(t)$ and $e_z(t)$ defined in \eqref{decaytest} for the different schemes. As expected by our asymptotic analysis, for the TAHO scheme the absolute errors $e_{u,z}(t)$ decay faster as the time increases. This result is confirmed in Table \ref{tab-nlin_2}, where we compute the decay parameters $\gamma$ of absolute errors previously plotted.
}
\label{fig-nlin2}
\end{figure}

\begin{table}[ht]
\begin{center}
\begin{tabular}{|c|c|c|c|c|c|c|c|c|}\hline
\mbox{scheme} & $C_u$ &  $\gamma_u$ & $C_z$ & $\gamma_z$ \\ \hline
 STD & 0.013797 & 0.374708 & 0.010744 &  0.341554 \\ \hline
 ROE & 0.004874 & 0.333634 & 0.007850 &  0.439996  \\ \hline
 TAHO & 0.111380 & 1.151517 & 0.495480 &  1.451030  \\ \hline
\end{tabular}
\caption{Non-Linear Case Test with $F'(0)\ne 0$, see section \ref{nLinT2}. Evaluation of constants $\gamma$ and $C$ for $e_u(t)= C_u t^{-\gamma_u}$ and $e_z(t)= C_z t^{-\gamma_z}$ defined in \eqref{decaytest}. For standard approximation STD and ROE, the absolute error decays as $e_{u,z}(t)\approx O(t^{-1/2})$; while, for the TAHO scheme it improves of $t^{-1/2}$.
}\label{tab-nlin_2}
\end{center}
\end{table}

\subsection{Results for the $3\times 3$ system.}\label{nLinT3x3}
As initial data, we take the smooth
function $u_0$ defined by
\[ u_0(x)=\chi_{[-1,1]} \, {\rm exp}\left(1- \frac{1}{1-x^2}\right).
\]
Then we set $f_0(x)=M(u_0(x))$. We know that in this case one has
$u(x,t) \in [0,1]$ for all $(x,t) \in \R \times [0,+\infty[$. We
choose $a=1$, $\lambda=2.1$, $\beta=(\alpha-a)/\lambda=0.1$. The discretization
parameters are $\Delta x=0.1$, $\rho=\displaystyle \frac{1}{2
  \lambda}$, which satisfy all the monotonicity requirements, see
proposition \ref{mono-33}.

The numerical results show a better performance of the TAHO scheme. In
Figures \ref{fig-nlin3-u}, \ref{fig-nlin3-z1}-(a)-(b),
\ref{fig-nlin3-z2}-(a)-(b),  we plot the solutions $u$, $Z_1$ and $Z_2$ respectively, obtained by the the STD, ROE and TAHO schemes at final time $T$, as
well as the exact (reference) solution. The solution given by applying the TAHO scheme follows much better than the other the benchmark curve. We stress on that the numerical solutions are computed with quite big step $\Dx=0.1$. 
\\
Then in Figure \ref{fig-nlin3-err}-(a)-(b), we plot the time evolution
of the $l^\infty$ errors $e_u(t)$ and $e_z(t)$. They show how for the TAHO scheme both errors decay as time increases more quickly than other. 
This result is also confirmed by Table \ref{tab-nlin3}, where the
values of $\gamma$ and $C$ are computed. Looking at the different
values of $\gamma$, it is clear that for the TAHO approximation the
decay velocity of the absolute error improves of $t^{-1/2}$ on the
previous schemes.
\begin{center}
\begin{table}[ht]
\begin{tabular}{|c|c|c|c|c|}\hline
\mbox{scheme} & $C_u$ &  $\gamma_u$ & $C_z$ & $\gamma_z$ \\ \hline
STD & 0.0052  &  0.54   & 0.0064 &  1.1 \\ \hline
ROE & 0.0027 & 0.66 & 0.0036  &  1.2 \\ \hline
TAHO & 0.006 & 1 &        0.012 & 1.62\\ \hline
\end{tabular}
\caption{The $3 \times 3$ system with $F^\prime(0)\ne0$, section
  \ref{nLinT3x3}. Evaluation of
  constants $\gamma$ and $C$ for $e_u(t)= C_u t^{-\gamma_u}$ and
  $e_z(t)= C_z t^{-\gamma_z}$ defined in \eqref{decaytest}. For STD and ROE
  approximations, the numerical results show that the
  absolute error decays as $e_u(t)= O( t^{-1/2})$ and $e_z(t)= O(
  t^{-1})$; while, for TAHO's it improves of $t^{-1/2}$.}
\label{tab-nlin3}
\end{table}
\end{center}    
\begin{figure}[ht]
\begin{tabular}{cc}
\includegraphics[height=6.0cm,angle=-90]{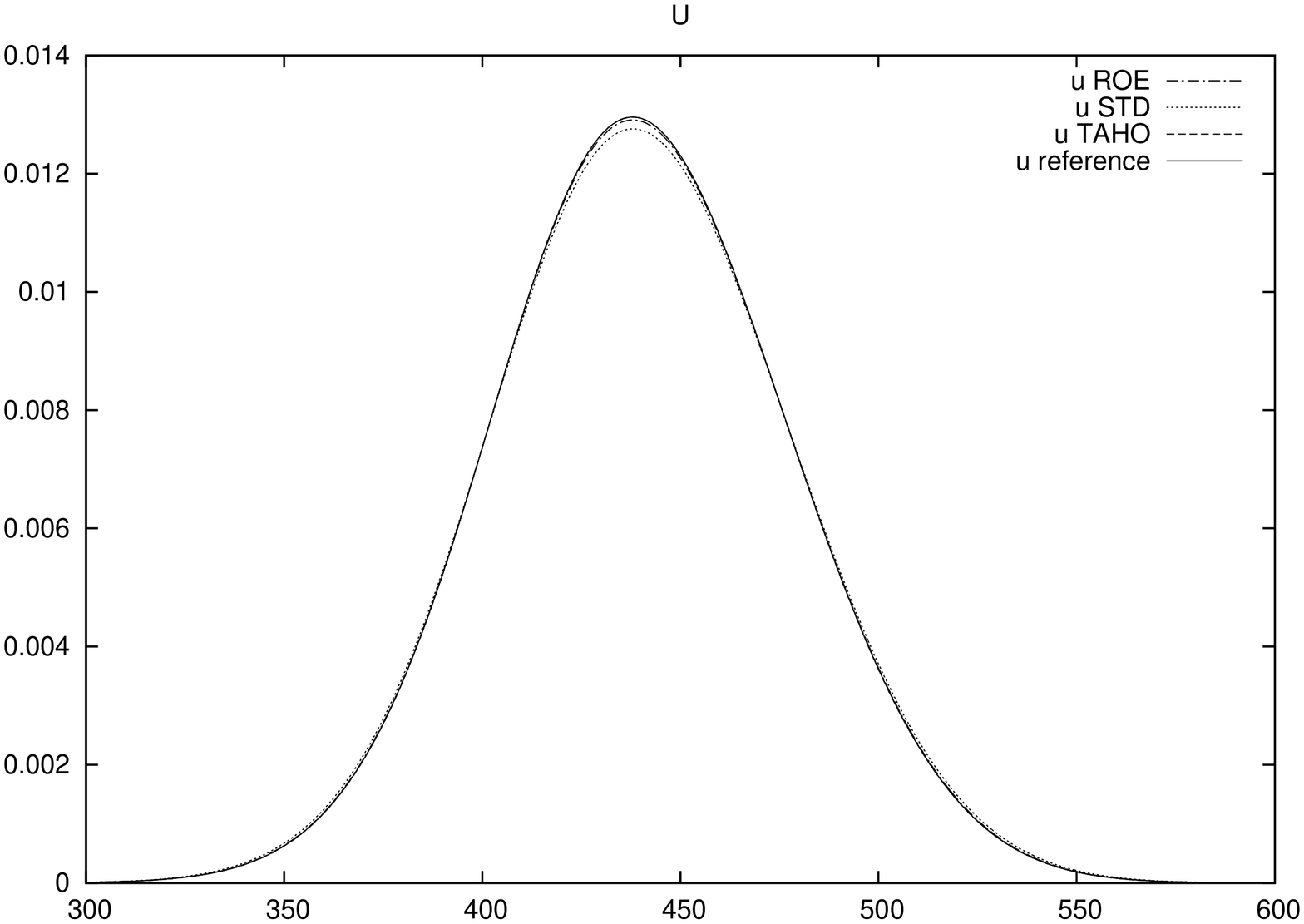}&
\includegraphics[height=6.0cm,angle=-90]{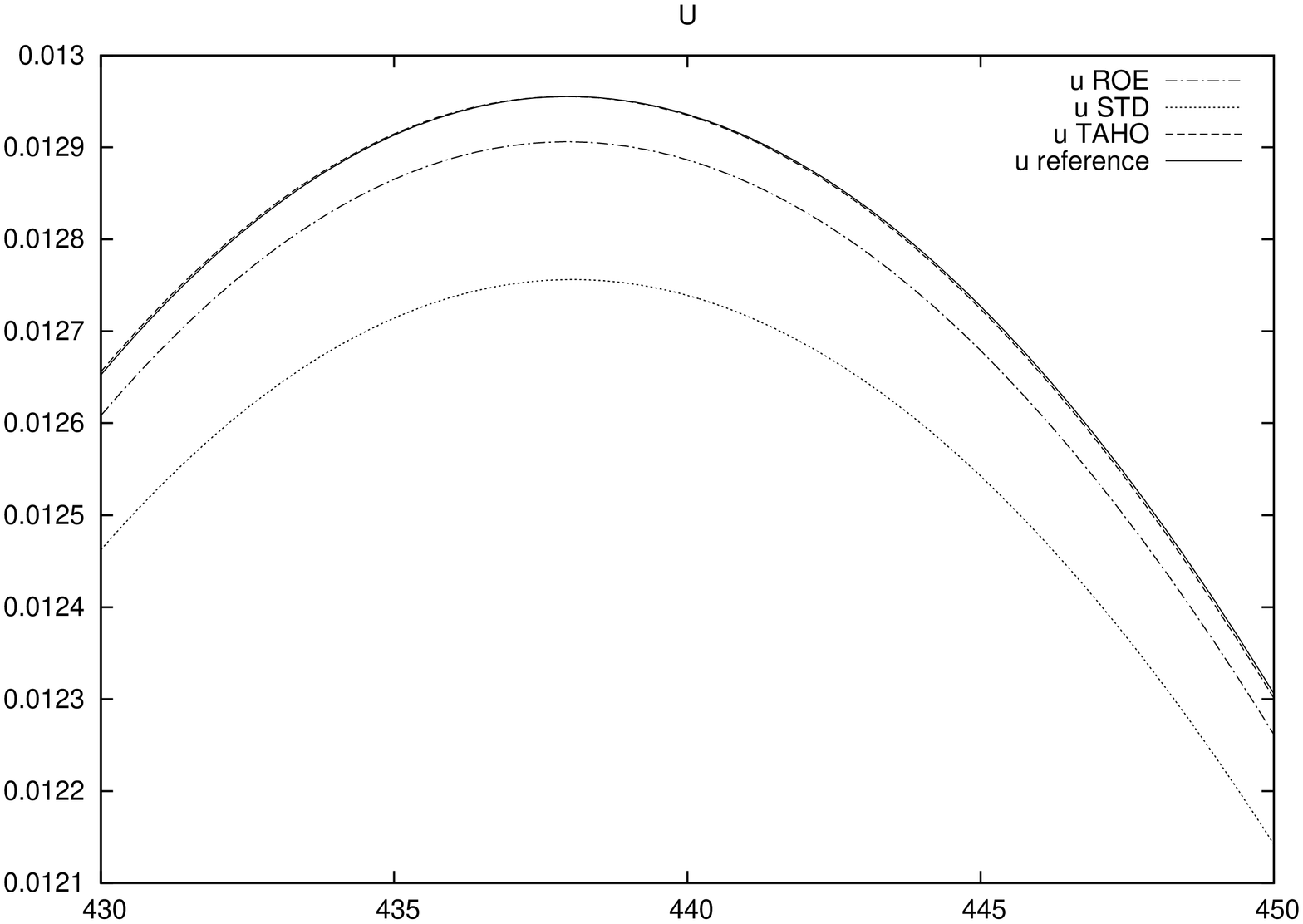}
\end{tabular}
\caption{The $3 \times 3$ system with $F^\prime(0)\not =0$. Left: $u$ component at
  final time $T$. Right: detail. The reference and the TAHO computed solutions cannot be distinguished.}
\label{fig-nlin3-u}
\end{figure}

\begin{figure}[ht]
\begin{tabular}{cc}
\includegraphics[height=6.0cm,angle=-90]{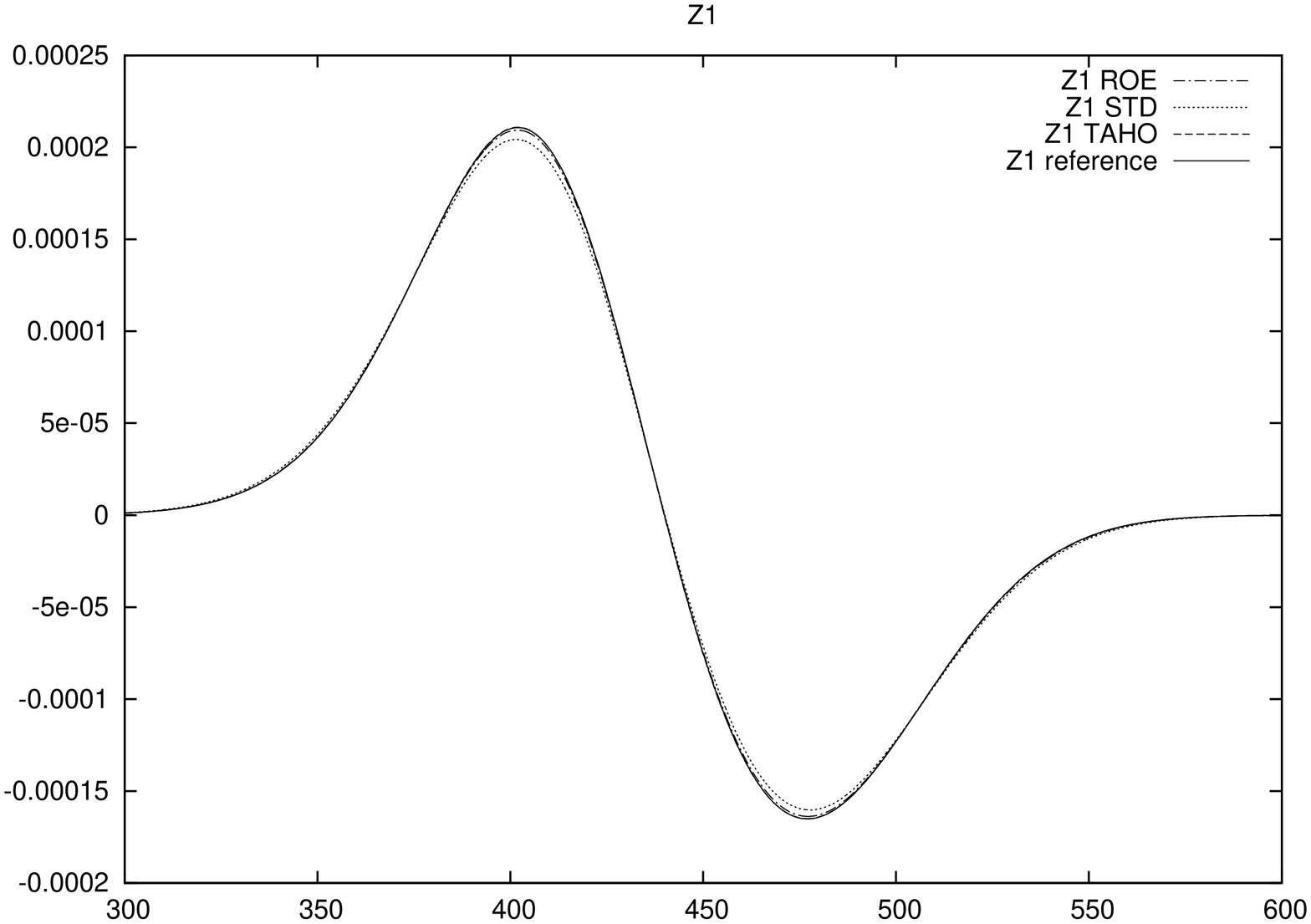}
&
\includegraphics[height=6.0cm,angle=-90]{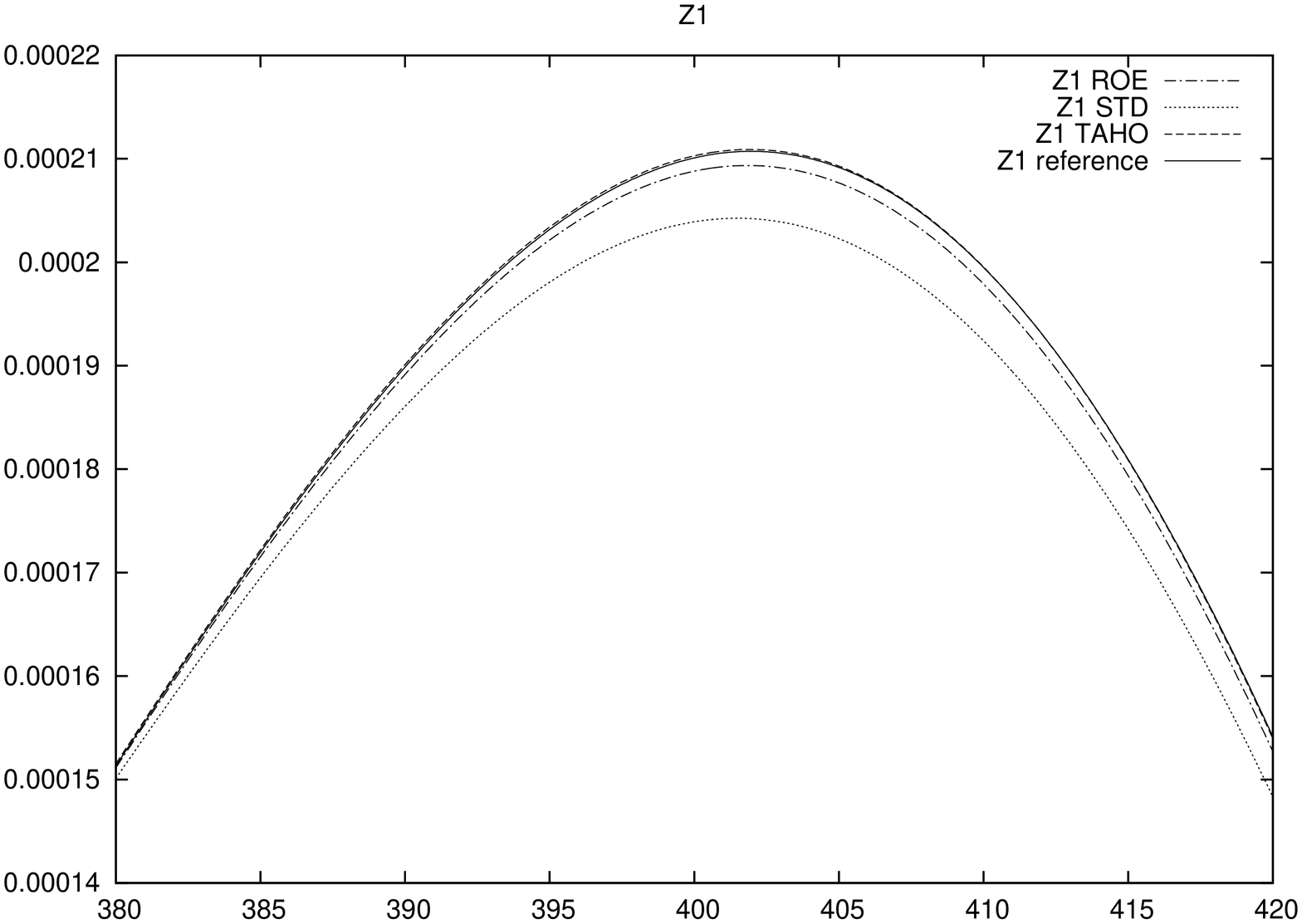}\\
\\[15pt] 
(a) & (b)
\end{tabular}
\caption{The $3 \times 3$ system with $F^\prime(0)\not =0$. Left: $Z_1$ component at
  final time $T$. Right: detail. The reference and the TAHO computed solutions cannot be distinguished.}
\label{fig-nlin3-z1}
\end{figure}

\begin{figure}[ht]
\begin{tabular}{cc}
\includegraphics[height=6.0cm,angle=-90]{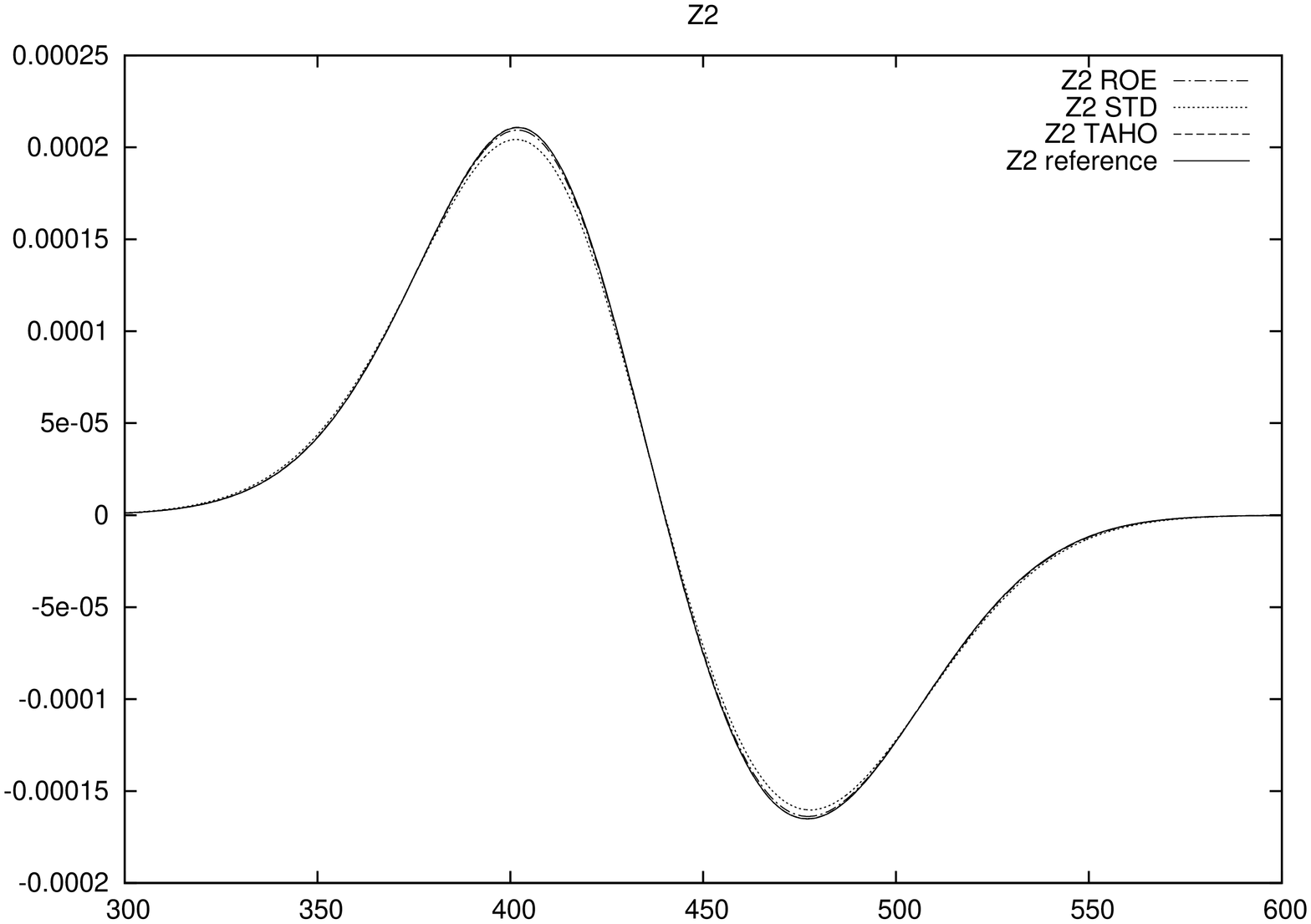}
&
\includegraphics[height=6.0cm,angle=-90]{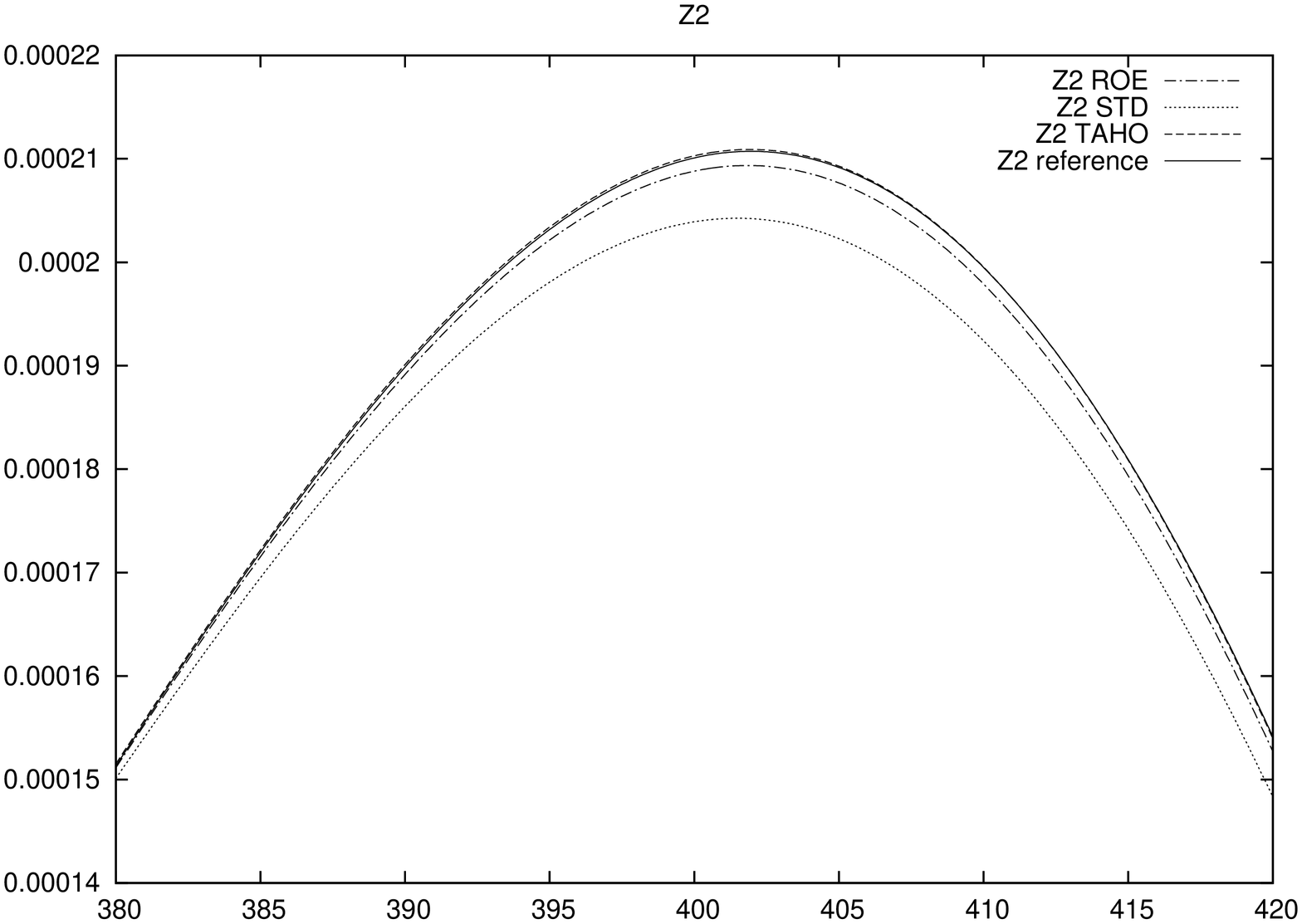}\\
\\[15pt] 
(a) & (b)
\end{tabular}
\caption{The $3 \times 3$ system with $F^\prime(0)\not =0$. Left: $Z_2$ component at
  final time $T$. Right: detail. The reference and the TAHO
  computed solutions cannot be distinguished.}
\label{fig-nlin3-z2}
\end{figure}

\begin{figure}[ht]
\begin{tabular}{cc}
\includegraphics[height=6.0cm,angle=-90]{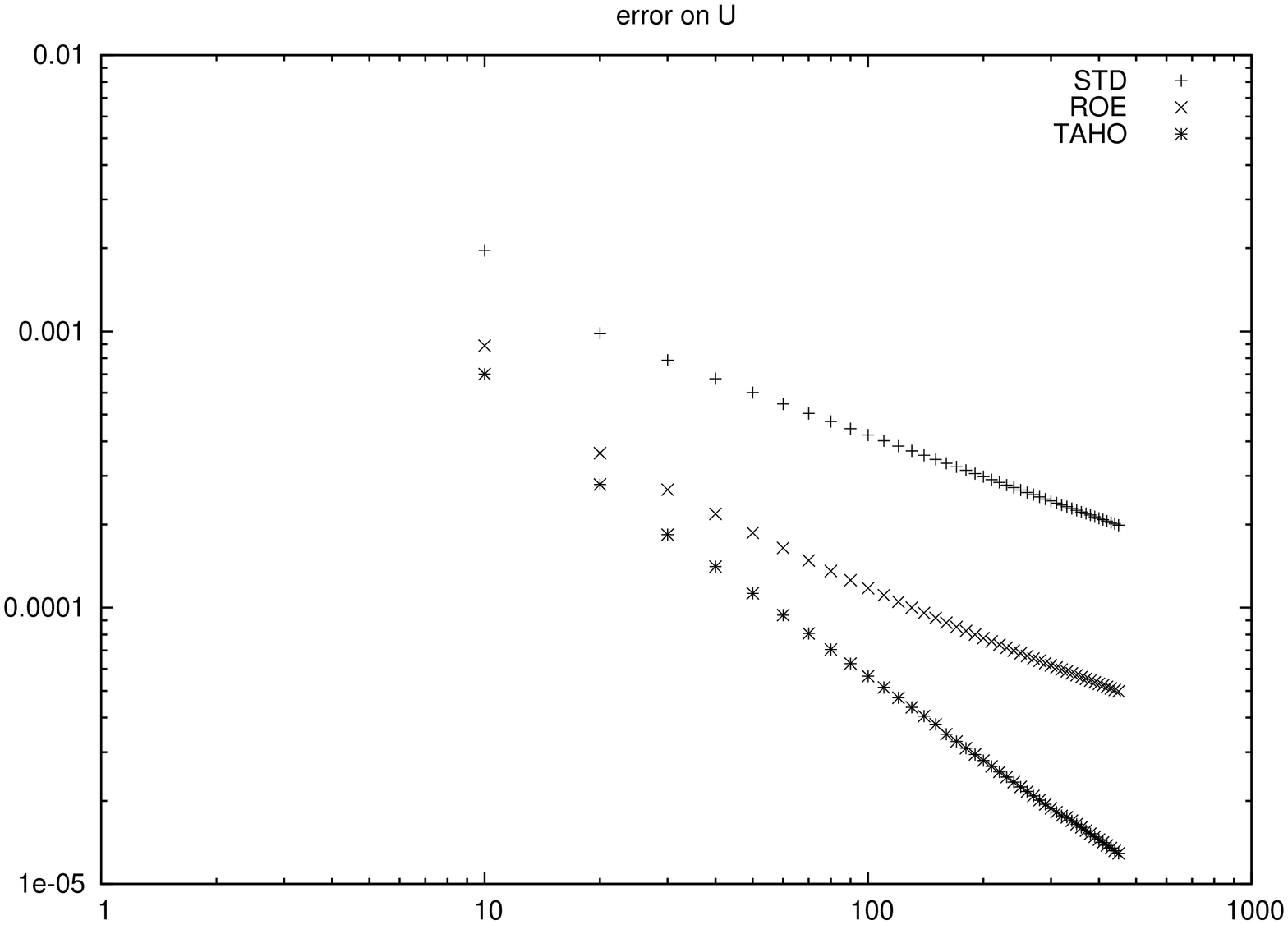}&
\includegraphics[height=6.0cm,angle=-90]{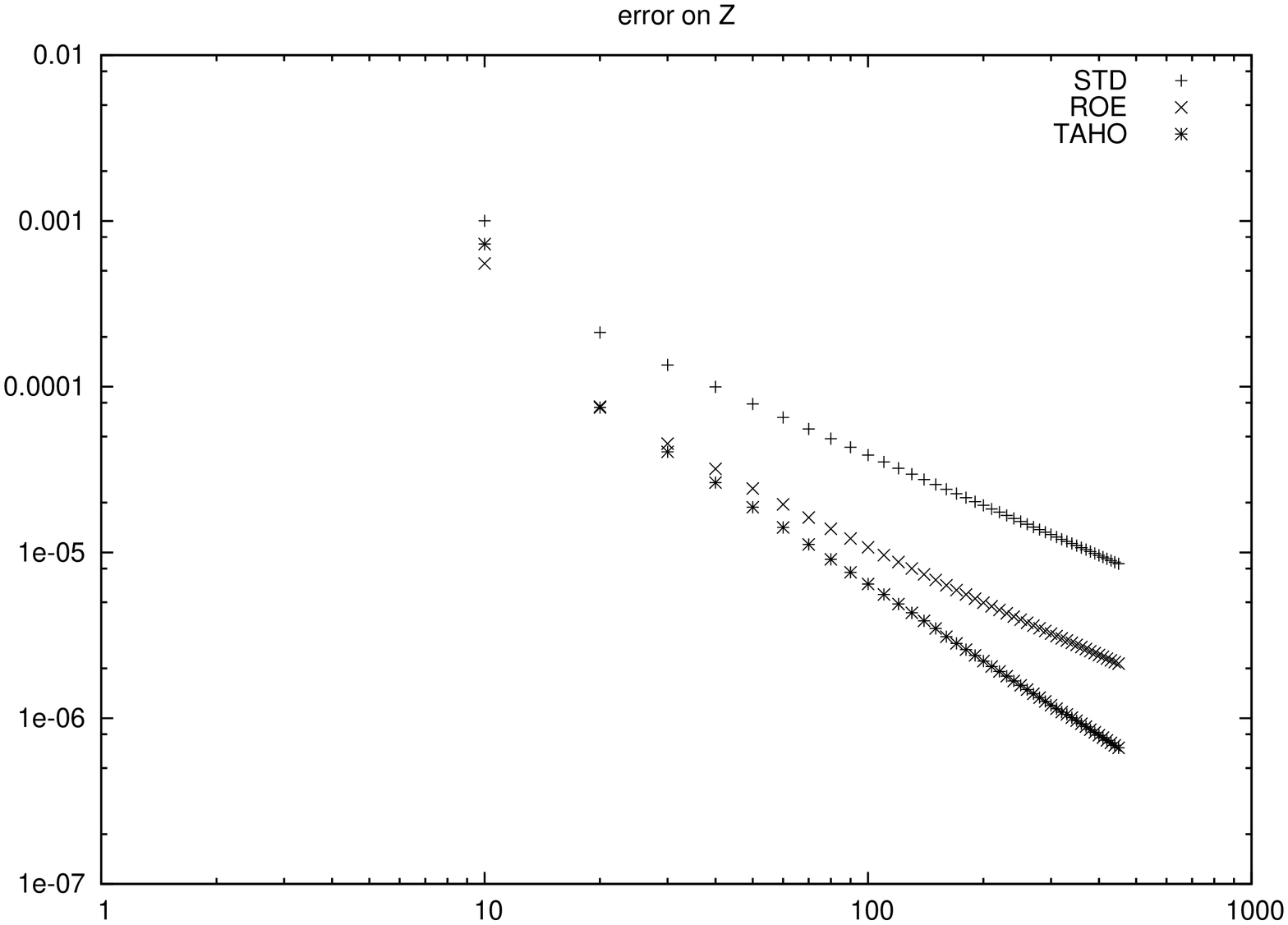}
\end{tabular}
\\[15pt] 
\caption{The $3 \times 3$ system with $F^\prime(0)\not =0$. Absolute
  errors for $u$ component (left) and $Z$ components (right) with
  respect to time.}
\label{fig-nlin3-err}
\end{figure}

\end{document}